\documentclass{article}

 \usepackage[final]{neurips_2019}

\usepackage[utf8]{inputenc} 
\usepackage[T1]{fontenc}    
\usepackage{hyperref}       
\usepackage{url}            
\usepackage{booktabs}       
\usepackage{amsfonts}       
\usepackage{nicefrac}       
\usepackage{microtype}      
\usepackage{microtype}
\usepackage{graphicx}
\usepackage{float}
\usepackage{soul}
\usepackage{caption}
\usepackage{bbding}
\usepackage{float}
\usepackage{sidecap}
\usepackage{enumitem}
\usepackage{amsmath,amssymb,amsthm}
\usepackage{color}
\usepackage{algorithm,algorithmic}
\usepackage{amsfonts}
\usepackage{subfigure}

\usepackage{hyperref}
\usepackage{amsmath}
\usepackage{booktabs} 
\newtheorem{theorem}{Theorem}
 \newtheorem{assumption}{Assumption}
\newtheorem{lemma}[theorem]{Lemma}

\newtheorem{definition}{Definition}[section]
\usepackage{hyperref}

\title{Learning Mean-Field Games}

\author{Xin Guo\\
University of California, Berkeley\\
\texttt{xinguo@berkeley.edu}\\
\And
Anran Hu\\
University of California, Berkeley\\
  \texttt{anran\_hu@berkeley.edu}\\
\And
Renyuan Xu\\
University of California, Berkeley\\
  \texttt{renyuanxu@berkeley.edu}\\
\And
Junzi Zhang\\
Stanford University \\
\texttt{junziz@stanford.edu}
}

\begin{document}

\date{}
\maketitle

\begin{abstract}
This paper presents a general mean-field game (GMFG) framework for simultaneous learning and decision-making in stochastic games with a large population. It first establishes the existence of a unique Nash Equilibrium  to this GMFG, and explains that naively combining 
Q-learning with the fixed-point approach in classical MFGs yields unstable algorithms. It then proposes a Q-learning algorithm with Boltzmann policy (GMF-Q), with analysis of convergence property  and computational  complexity. 
The experiments  on repeated Ad auction problems  demonstrate that this GMF-Q algorithm is
 efficient and robust in terms of convergence and learning accuracy. Moreover, its performance is superior in convergence, stability, and learning ability, when compared with existing algorithms for multi-agent reinforcement learning.
 \end{abstract}
\section{Introduction}
\label{introduction}

\paragraph{Motivating example.} This paper is motivated by the following  Ad auction problem for an advertiser. An Ad auction is a stochastic game on an Ad exchange platform among a large number of players, the advertisers. 
In between the time a web user requests a page and the time the page is displayed, usually within a millisecond,  a Vickrey-type of second-best-price  auction is run to incentivize 
interested advertisers  to bid for an Ad slot to display advertisement.
Each advertiser has limited information before each bid:  first, her own {\it valuation} for a slot depends on an unknown conversion of clicks for the item; secondly, she, should she win the bid,  only knows the reward {\textit{after}} the user's activities on the website are finished. In addition, she has a budget constraint in this repeated auction. 

The question is, how should she bid in this online sequential repeated game when there is a {\textit{large}} population of bidders competing on the Ad platform, with {\textit{unknown}} distributions of the conversion of clicks and rewards? 

Besides the Ad auction, there are many real-world problems involving a large number of players and unknown systems.  Examples include massive multi-player online role-playing games \cite{MMORPG}, high frequency tradings \cite{algo_trade_mfg}, and the sharing economy \cite{sharing_eco}.

\paragraph{Our work.}

Motivated by these problems, we consider a general framework  of  simultaneous learning and decision-making in stochastic games with a large population. We formulate a general mean-field-game (GMFG) with incorporation of   action distributions, (randomized) relaxed policies, and with unknown rewards and dynamics.  This {\color{black}general} framework can also be viewed as a generalized version of MFGs of McKean-Vlasov type \cite{MKV}, which is a different paradigm from the classical MFG. It is also beyond the scope of the existing Q-learning framework for Markov decision problem (MDP) with unknown distributions, as MDP is technically equivalent to a single player stochastic game.

On the theory front, this {\color{black}general} framework differs from all existing MFGs. We establish under appropriate technical conditions,  the existence and uniqueness of the 
Nash equilibrium (NE)  to this GMFG.  On the computational front, we show  that naively combining 
Q-learning with the three-step fixed-point approach in classical MFGs yields unstable algorithms. We then propose a Q-learning algorithm with Boltzmann policy (GMF-Q), establish its convergence property and analyze its computational complexity.
Finally, we apply this GMF-Q algorithm to  the Ad auction problem, where
this GMF-Q algorithm demonstrates its
 efficiency and robustness in terms of convergence and learning. Moreover, 
 its performance is superior, when compared with existing algorithms for multi-agent reinforcement learning for convergence, stability, and learning accuracy.

\paragraph{Related works.} On learning large population games with mean-field approximations, \cite{YYTXZ2017} focuses on inverse reinforcement learning  for MFGs without decision making, \cite{YLLZZW2018} studies an MARL problem with a {\color{black}first-order} mean-field approximation term modeling the interaction between one player and all the other finite players, and \cite{KC2013} and \cite{YMMS2014} consider model-based adaptive learning for MFGs {\color{black}in specific models (\textit{e.g.}, linear-quadratic and oscillator games)}. {\color{black}More recently, \cite{Manymany} studies the local convergence of actor-critic algorithms on finite time horizon MFGs, and \cite{rl_mfg_local} proposes a policy-gradient based algorithm and analyzes the so-called local NE for reinforcement learning in infinite time horizon MFGs.}  For learning large population games without mean-field approximation, see \cite{MARL_literature2, MARL_literature1} and the references therein.  
In the specific topic of learning auctions with a large number of advertisers, \cite{CRZMWYG2017} and \cite{JSLGWZ2018} explore reinforcement learning techniques to search for social optimal solutions with real-word data, and \cite{IJS2011} uses MFGs to model the auction system with unknown conversion of clicks within a Bayesian framework.


{\color{black}However, none of these works consider the problem of simultaneous learning and decision-making in a general MFG framework. Neither do they establish the existence and uniqueness of  the {\color{black}(global)} NE, nor do they present  model-free learning algorithms with complexity analysis and  convergence to the NE.} Note that in principle, global results are harder to obtain compared to local results.


\section{Framework of General MFG (GMFG)}\label{n2mfg}

\subsection{Background: classical $N$-player Markovian game and MFG}\label{classical}
Let us first recall the classical $N$-player game. There are $N$ players in a game.   
 At each step $t$, the state of  player $i   \ \ (=1, 2, \cdots, N)$ is  $s^i_t\in\mathcal{S}\subseteq\mathbb{R}^d$ and she takes an action $a^i_t\in\mathcal{A}\subseteq\mathbb{R}^p$. Here $d, p$ are positive integers, and $\mathcal{S}$ and  $\mathcal{A}$ are  compact (for example,  finite) state space and action space, respectively. 
 Given the current state profile of $N$-players ${\bf s}_t=(s^1_t,\dots,s^N_t)\in\mathcal{S}^N$ and the action $a^i_t$,   player $i$ will receive a reward $r^i({\bf s}_t, a^i_t)$ and her state will change to $s^i_{t+1}$ according to a transition probability function $P^i({\bf s}_t, a^i_t)$.

 A Markovian game further restricts the admissible policy/control for player $i$ to be of the form
 $a^i_t=\pi^i_t({\bf s}_t)$. That is,  $\pi^i_t:\mathcal{S}^N\rightarrow \mathcal{P}(\mathcal{A})$  maps each state profile ${\bf s}\in\mathcal{S}^N$ to a randomized action, with {\color{black}$\mathcal{P}(\mathcal{X})$  the space of probability measures on space $\mathcal{X}$}. 
The accumulated reward (a.k.a. the value function) for player $i$, given the initial state profile ${\bf s}$ and the policy profile sequence $\pmb{\pi} :=\{\pmb{\pi}_t\}_{t=0}^{\infty} $ with $\pmb{\pi}_t=(\pi^1_t,\dots,\pi^N_t)$, is then defined as
\begin{eqnarray}\label{game}
V^i({\bf s},\pmb{\pi}):=\mathbb{E}\left[\sum_{t=0}^{\infty}\gamma^t r^i({\bf s}_t,a^i_t)\Big| {\bf s}_0={\bf s}\right],
\end{eqnarray}
where $\gamma\in(0,1)$ is the discount factor,  $a^i_t\sim \pi^i_t({\bf s}^t)$, and $s^i_{t+1}\sim P^i({\bf s}_t, a_t^i)$.
The goal of each player is to maximize her value function over all admissible policy sequences.   

In general, this type of stochastic $N$-player game is notoriously hard to analyze, especially when $N$  is large {\color{black}\cite{PR05}}. Mean field game (MFG),  pioneered by \cite{HMC2006} and \cite{LL2007} {\color{black}in the continuous settings and later developed  in \cite{MFG_n_conv, MFG_gomes, MFG_binact, MFG_discrete_time, MFG_discrete_time2} for discrete settings}, provides an ingenious and  tractable aggregation approach to approximate the otherwise challenging $N$-player stochastic games. 
The basic idea for an MFG goes as follows. Assume all players are  identical, indistinguishable and interchangeable, when $N\to \infty$, one can view the limit of other players' states ${\bf s}_t^{-i}=(s_t^1,\dots,s_t^{i-1},s_t^{i+1},\dots,s_t^N)$ as a population state distribution {\color{black}$\mu_t$ with $\mu_t(s):=\lim_{N \rightarrow \infty}\frac{\sum_{j=1, j\neq i}^N \textbf{I}_{s_t^j=s}}{N}$}.\footnote{{\color{black}Here the indicator function $\textbf{I}_{s_t^j=s}=1$ if $s_t^j=s$ and $0$ otherwise.}} Due to the homogeneity of the players, one can then focus on a single (representative) player. That is, in an MFG, one may consider   instead   the following optimization problem,
\[
\begin{array}{ll}
\text{maximize}_{\pmb{\pi}} & V(s,\pmb{\pi},\pmb{\mu}):=\mathbb{E}\left[\sum\limits_{t=0}^\infty \gamma^t r(s_t,a_t,\mu_t)|s_0=s\right]\\
\text{subject to} & s_{t+1}\sim P(s_t,a_t,\mu_t), \quad a_t\sim \pi_t(s_t,\mu_t),
\end{array}
\]
where $\pmb{\pi}:={\{\pi_t\}_{t=0}^{\infty}}$ denotes the policy sequence and $\pmb{\mu} := \{\mu_t\}_{t=0}^\infty$  the distribution flow.
In this MFG setting, at time $t$, after the representative player chooses her action $a_t$ according to some policy $\pi_t$, she will receive reward $r(s_t,a_t,\mu_t)$ and her state will evolve under a \textit{controlled stochastic dynamics} of a mean-field type $P(\cdot|s_t,a_t,\mu_t)$.
Here the policy $\pi_t$ depends on both the current state $s_t$ and the current population state distribution $\mu_t$ such that $\pi:\mathcal{S}\times \mathcal{P}(\mathcal{S})\rightarrow \mathcal{P}(\mathcal{A})$.

\subsection{General MFG (GMFG)}\label{mfg-set-up}
In the classical MFG setting, the reward and the dynamic for each player are known. They depend only on $s_t$ the state of the player,  $a_t$ the action of this particular player, and $\mu_t$ the population state distribution. In contrast, in the motivating auction example, the reward and the dynamic  are unknown; they  rely on the actions of {\it all} players,  as well as on $s_t$ and $\mu_t$.

We therefore define the following  general MFG (GMFG) framework.  At time $t$, after the representative player chooses her action $a_t$ according to some policy  $\pi:\mathcal{S}\times \mathcal{P}(\mathcal{S})\rightarrow \mathcal{P}(\mathcal{A})$, she will receive a reward $r(s_t,a_t,\mathcal{L}_t)$ and her state will evolve according to $P(\cdot|s_t,a_t,\mathcal{L}_t)$, where $r$ and $P$ are possibly unknown. The objective of the player is to solve the following control problem:
\begin{equation}\label{mfg}
\begin{array}{ll}
\text{maximize}_{\pmb{\pi}} & V(s,\pmb{\pi},\pmb{\mathcal{L}}):=\mathbb{E}\left[\sum\limits_{t=0}^\infty \gamma^t r(s_t,a_t,{\color{black}\mathcal{L}_t})|s_0=s\right]\\
\text{subject to} & s_{t+1}\sim P(s_t,a_t,{\color{black}\mathcal{L}_t}),\quad a_t\sim \pi_t(s_t,\mu_t).
\end{array}\tag{GMFG}
\end{equation}
Here, $\pmb{\mathcal{L}}:=\{\mathcal{L}_t\}_{t=0}^{\infty}$, with
$\mathcal{L}_t=\mathbb{P}_{s_t,a_t}\in \mathcal{P}(\mathcal{S}\times \mathcal{A})$  the joint distribution  of the state and the action (\textit{i.e.}, the \text{population state-action pair}). $\mathcal {L}_t$  has  marginal distributions $\alpha_t$ for the population action  and $\mu_t$ for the population state. Notice that $\{\mathcal{L}_t\}_{t=0}^{\infty}$ could depend on time. Namely, an infinite time horizon MFG could still have time-dependent NE solution due to the mean information process (game interaction) in the MFG. This is fundamentally different from the theory of single-agent MDP where the optimal control, if exists uniquely, would be time independent in an infinite time horizon setting.

In this framework, we adopt the well-known Nash Equilibrium (NE) for analyzing stochastic games.

\begin{definition}[NE for GMFGs]\label{nash2} 
In \eqref{mfg}, a player-population profile $(\pmb{\pi}^{\star},\pmb{\mathcal{L}}^{\star}):=(\{\pi_t^\star\}_{t=0}^{\infty},\{\mathcal{L}_t^\star\}_{t=0}^{\infty})$
  is called an NE if 
\begin{enumerate}
    \item (Single player side) Fix $\pmb{\mathcal{L}}^{\star}$, for any policy sequence $\pmb{\pi}:=\{\pi_t\}_{t=0}^{\infty}$ and initial state $s\in \mathcal{S}$, 
\begin{equation}
V\left(s,\pmb{\pi}^{\star},\pmb{\mathcal{L}}^{\star}\right)\geq V\left(s,\pmb{\pi},\pmb{\mathcal{L}}^{\star}\right).
\end{equation}
\item  (Population side) $\mathbb{P}_{s_t,a_t}= {\mathcal{L}_t^{\star}}$ for all $t\geq 0$, where $\{s_t,a_t\}_{t=0}^{\infty}$ is the dynamics under the policy sequence $\pmb{\pi}^{\star}$ starting from $s_0 \sim \mu_0^{\star}$, with $a_t\sim\pi_t^\star(s_t,{\color{black}\mu_t^{\star}})$, $s_{t+1}\sim P(\cdot|s_t,a_t,{\color{black}\mathcal{L}_t^\star})$, and $\mu_t^{\star}$ being the population state marginal of $\mathcal{L}_t^\star$.
\end{enumerate}
\end{definition}
The single player side condition captures the optimality of $\pmb{\pi}^\star$, when the population side is fixed. The population side condition ensures the ``consistency'' of the solution: it guarantees that the state and action distribution flow of the single player does match the population state and action sequence $\pmb{\mathcal{L}}^{\star}$.

\subsection{Example: GMFG for the repeated auction} 
\label{section:example}
Now, consider the repeated Vickrey  auction  with a budget constraint in Section \ref{introduction}. Take a representative advertiser in the auction.
Denote $s_t \in \{0,1,2,\cdots,s_{\max}\}$ as the budget of this player at time $t$, where $s_{\max}\in \mathbb{N}^+$ is  the maximum budget allowed on the Ad exchange with a unit bidding price.
Denote  $a_t$ as the bid price submitted by this player  and $\alpha_t$ as the bidding/(action) distribution of the population. The reward for this  advertiser with bid $a_t$ and budget $s_t$ is 
  \begin{eqnarray}
\label{reward}
r_t = {\color{black}\textbf{I}_{w_t^M=1}}\left[(v_t-a^M_t)-(1+\rho){\color{black}\textbf{I}_{s_{t}<a^M_t}}(a^M_t-s_{t})\right].
\end{eqnarray}
Here $w^M_t$ takes values $1$ and $0$, with $w_t^M=1$ meaning this player winning the bid and $0$ otherwise.
The probability of winning the bid would depend on $M$, the index for the game intensity, and $\alpha_t$. {\color{black}(See discussion on $M$ in Appendix \ref{intensity}.)}
The conversion of clicks at time $t$ is $v_t$ and follows an unknown distribution. $a^M_t$ is the value of the second largest bid at time $t$, taking values from $0$ to $s_{\max}$, and  depends on both $M$ and $\mathcal{L}_t$.  Should the player win the bid, the reward $r_t$ consists of two parts, corresponding to the two terms in (\ref{reward}). The first term is the profit of wining the auction, as the winner only needs to pay for the second best bid {\color{black}$a_t^M$} in a Vickrey auction. The second term is the penalty of overshooting if the payment exceeds her budget, with a penalty rate $\rho$.
At each time $t$, the budget dynamics $s_t$ follows, \begin{eqnarray*}
    {s}_{t+1}=\left\{
                \begin{array}{ll}
                 s_t, \qquad & w^M_t \neq 1,\\
                 s_t - a^M_t, \qquad & w^M_t=1 \text{ and } a^M_t \leq s_t,\\
                 0, \qquad & w^M_t =1 \text{ and } a^M_t > s_t.
                \end{array}
              \right.
  \end{eqnarray*}
That is,  if this player does not win the bid, the budget will remain the same. If she wins and has enough money to pay, her budget will decrease from $s_t$ to $s_t - a^M_t$. However, if she wins  but does not have enough money, her budget will be $0$ after the payment and there will be a penalty in the reward function.
Note that in this game, both the rewards $r_t$ and the dynamics $s_t$ are unknown {\it a priori}.

In practice, one often modifies the dynamics of $s_{t+1}$  with a non-negative random budget fulfillment $\Delta({s}_{t+1})$ after the auction clearing \cite{GKP2012}, such that 
$\hat{s}_{t+1} = {s}_{t+1} + \Delta({s}_{t+1})$.
One may see some particular choices of $\Delta({s}_{t+1})$ in the experiment section (Section \ref{experiments}).

\section{Solution for GMFGs}\label{MFG_basic}
We now establish the existence and uniqueness of the NE  to (\ref{mfg}), by
generalizing the classical fixed-point approach for MFGs to this GMFG setting. (See \cite{HMC2006} and \cite{LL2007} for the classical case).  It consists of three steps. 
\paragraph{Step A.}
Fix $\pmb{\mathcal{L}}:=\{\mathcal{L}_t\}_{t=0}^{\infty}$,   (\ref{mfg}) becomes the  classical optimization problem.  Indeed, with $\pmb{\mathcal{L}}$ fixed, the population state distribution sequence $\pmb{\mu}:=\{\mu_t\}_{t=0}^{\infty}$ is also fixed, hence the space of admissible policies is reduced to the single-player case. Solving (\ref{mfg}) is now reduced to finding a policy sequence
 $\pi_{t,\pmb{\mathcal{L}}}^\star\in \Pi:=\{\pi\,|\,\pi:\mathcal{S}\rightarrow\mathcal{P}(\mathcal{A})\}$
 over all admissible  $\pmb{\pi}_{\pmb{\mathcal{L}}}=\{\pi_{t,\pmb{\mathcal{L}}}\}_{t=0}^{\infty}$,
   to maximize 
 \[
\begin{array}{ll}
  V(s,\pmb{\pi}_{\pmb{\mathcal{L}}}, \pmb{\mathcal{L}}):= &\mathbb{E}\left[\sum\limits_{t=0}^{\infty}\gamma^tr(s_t,a_t,\mathcal{L}_t)|s_0=s\right],\\
\text{subject to}& s_{t+1}\sim P(s_t,a_t,\mathcal{L}_t),\quad a_t\sim\pi_{t,\pmb{\mathcal{L}}}(s_t).
\end{array}
\]
Notice that with $\pmb{\mathcal{L}}$ fixed, one can safely suppress the dependency on $\mu_t$ in the admissible policies. 
Moreover,  given this fixed $\pmb{\mathcal{L}}$ sequence and the solution $\pmb{\pi}_{\pmb{\mathcal{L}}}^\star:=\{\pi_{t,\pmb{\mathcal{L}}}^\star\}_{t=0}^{\infty}$, one can 
define a mapping from the fixed population distribution sequence $\pmb{\mathcal{L}}$ to an arbitrarily chosen optimal randomized policy sequence. That is, 
$$\Gamma_1:\{\mathcal{P}(\mathcal{S}\times \mathcal{A})\}_{t=0}^{\infty}\rightarrow\{\Pi\}_{t=0}^{\infty}, $$
such that  $\pmb{\pi}_{\pmb{\mathcal{L}}}^\star=\Gamma_1(\pmb{\mathcal{L}})$.
Note that this $\pmb{\pi}_{\pmb{\mathcal{L}}}^\star$ sequence satisfies the single player side condition in Definition \ref{nash2} for the population state-action pair sequence $\pmb{\mathcal{L}}$. That is, 
$
V\left(s,\pmb{\pi}_{\pmb{\mathcal{L}}}^\star,{\color{black}\pmb{\mathcal{L}}}\right)\geq V\left(s,\pmb{\pi},{\color{black}\pmb{\mathcal{L}}}\right),$
for any policy sequence $\pmb{\pi} = \{\pi_t\}_{t=0}^{\infty}$ and any initial state $s\in\mathcal{S}$. 

As in the MFG literature \cite{HMC2006}, a feedback regularity condition is needed for analyzing Step A. 
\begin{assumption}\label{policy_assumption}
There exists a constant $d_1\geq 0$,
 such that for any $\pmb{\mathcal{L}}, \pmb{\mathcal{L}}^{\prime} \in \{\mathcal{P}(\mathcal{S}\times\mathcal{A})\}_{t=0}^{\infty}$, 
\begin{equation}\label{Gamma1_lip}
D(\Gamma_1(\pmb{\mathcal{L}}),\Gamma_1( \pmb{\mathcal{L}}^{\prime})) \leq d_1\mathcal{W}_1(\pmb{\mathcal{L}}, \pmb{\mathcal{L}}^{\prime}),
\end{equation}
where 
\begin{equation}
\begin{split}
D(\pmb{\pi},\pmb{\pi}^{\prime})&:=\sup_{s\in\mathcal{S}}\mathcal{W}_1(\pmb{\pi}(s),\pmb{\pi}^{\prime}(s))=\sup_{s\in\mathcal{S}}\sup_{t\in\mathbb{N}}W_1(\pi_t(s),\pi_t'(s)),\\
\mathcal{W}_1(\pmb{\mathcal{L}}, \pmb{\mathcal{L}}^{\prime})&:=\sup_{t\in\mathbb{N}}W_1(\mathcal{L}_t,\mathcal{L}_t'),
\end{split}
\end{equation}
and $W_1$ is the $\ell_1$-Wasserstein distance between probability measures \cite{metrics_prob,  COT_cuturi, OT_ON}. 
\end{assumption}


 \paragraph{Step B.} Based on the analysis in Step A and $\pmb{\pi}_{\pmb{\mathcal{L}}}^\star=\{\pi_{t,\pmb{\mathcal{L}}}^\star\}_{t=0}^{\infty}$, update the initial sequence $\pmb{\mathcal{L}}$ to $ \pmb{\mathcal{L}}^{\prime}$ following the controlled dynamics $P(\cdot|s_t, a_t,\mathcal{L}_t)$.

Accordingly, for any admissible policy sequence $\pmb{\pi} \in \{\Pi\}_{t=0}^{\infty}$ and a joint population state-action pair sequence $\pmb{\mathcal{L}}\in \{\mathcal{P}(\mathcal{S}\times \mathcal{A})\}_{t=0}^{\infty}$, define a mapping $\Gamma_2:\{\Pi\}_{t=0}^{\infty}\times \{\mathcal{P}(\mathcal{S}\times\mathcal{A})\}_{t=0}^{\infty}\rightarrow \{\mathcal{P}(\mathcal{S}\times\mathcal{A})\}_{t=0}^{\infty}$ as follows: 
\begin{eqnarray}
\Gamma_2(\pmb{\pi},\pmb{\mathcal{L}}):= \pmb{\hat{\mathcal{L}}} = \{\mathbb{P}_{s_t,a_t}\}_{t=0}^{\infty},
\end{eqnarray}
where $s_{t+1}\sim \mu_t P(\cdot|\cdot,a_t,\mathcal{L}_t)$,  $a_t\sim \pi_t(s_t)$, $s_0 \sim \mu_0$, and $\mu_t$ is the population state marginal of $\mathcal{L}_t$.

One needs a standard assumption in this step.
\begin{assumption}\label{population_assumption}
There exist constants $d_2,~d_3\geq 0$, such that for any admissible policy sequences $\pmb{\pi},\pmb{\pi}^1,\pmb{\pi}^2$ and joint distribution sequences $\pmb{\mathcal{L}}, \pmb{\mathcal{L}}^{1}, \pmb{\mathcal{L}}^{2}$, 
\begin{equation}\label{Gamma2_lip1}
\mathcal{W}_1(\Gamma_2(\pmb{\pi}^1,\pmb{\mathcal{L}}),\Gamma_2(\pmb{\pi}^2,\pmb{\mathcal{L}})) \leq d_2 D(\pmb{\pi}^1,\pmb{\pi}^2), 
\end{equation}
\begin{equation}\label{Gamma2_lip2}
\mathcal{W}_1(\Gamma_2(\pmb{\pi},\pmb{\mathcal{L}}^1{\color{black})},\Gamma_2(\pmb{\pi},\pmb{\mathcal{L}}^2)) \leq d_3 \mathcal{W}_1(\pmb{\mathcal{L}}^1,\pmb{\mathcal{L}}^2).
\end{equation}
\end{assumption}
Assumption \ref{population_assumption} can be reduced to Lipschitz continuity  and boundedness of the transition dynamics $P$. 
(See the Appendix for more details.)

\paragraph{Step C.} Repeat Step A and Step B until $ \pmb{\mathcal{L}}^{\prime}$ matches $\pmb{\mathcal{L}}$.

This step is to take care of the population side condition. To ensure the convergence of
 the combined step A and step B,  it suffices if  $\Gamma:\{\mathcal{P}(\mathcal{S}\times\mathcal{A})\}_{t=0}^{\infty}\rightarrow \{\mathcal{P}(\mathcal{S}\times\mathcal{A})\}_{t=0}^{\infty}$ is a contractive mapping under the $\mathcal{W}_1$ distance, with $\Gamma(\pmb{\mathcal{L}}):=\Gamma_2(\Gamma_1(\pmb{\mathcal{L}}), \pmb{\mathcal{L}})$.  Then by the Banach fixed point theorem {\color{black}and the completeness of the related metric spaces}, there exists  a unique NE to the GMFG. 

 In summary, we have                                                                                                                    
 \begin{theorem}[Existence and Uniqueness of GMFG solution] \label{thm1} Given Assumptions \ref{policy_assumption} and \ref{population_assumption}, and assuming that $d_1d_2+d_3< 1$, there exists a unique NE  to \eqref{mfg}.
 
\end{theorem}

\section{RL Algorithms for {\color{black}(stationary)} GMFGs}\label{AIQL}


In this section, we design the computational algorithm for the GMFG. Since the reward and transition distributions are unknown, this is  simultaneously learning the system and finding the  NE of the game. 
We will focus on the case with finite state and action spaces, \textit{i.e.}, $|\mathcal{S}|,|\mathcal{A}|<\infty$. 
We will look for stationary (time independent)  NEs. Accordingly, we abbreviate $\pmb{\pi}:=\{\pi\}_{t=0}^{\infty}$ and $\pmb{\mathcal{L}}:=\{\mathcal{L}\}_{t=0}^{\infty}$ as $\pi$ and $\mathcal{L}$, respectively. This stationarity property enables developing appropriate time-independent Q-learning algorithm, suitable for an infinite time horizon game. Modification from the GMFG framework to this special stationary setting is straightforward, and is left to Appendix \ref{stat_mfg_app}. Note that the assumptions to guarantee the existence and uniqueness of GMFG solutions are slightly different between the stationary and non-stationary cases. For instance, one can compare \eqref{Gamma2_lip1}-\eqref{Gamma2_lip2} with \eqref{Gamma2_lip1_stat}-\eqref{Gamma2_lip2_stat}.

The algorithm consists of two steps, 
parallel to Step $A$ and Step $B$ in Section \ref{MFG_basic}.
\paragraph{Step 1: Q-learning with stability for fixed $\mathcal{L}$.}
With $\mathcal{L}$ fixed, it becomes a standard learning problem for an infinite horizon MDP. We will focus on the Q-learning algorithm \cite{Sutton, Ben_tutorial}.  

The Q-learning algorithm approximates the value iteration by stochastic approximation. At each step with the  state $s$ and an action $a$, the system reaches state $s'$ according to the controlled dynamics and the Q-function is updated according to
\begin{equation}\label{Q-learning}
Q_{\mathcal{L}}(s,a) \leftarrow ~(1-\beta_t(s,a))Q_{\mathcal{L}}(s,a)+ \beta_t(s,a) \left[r(s,a,\mathcal{L})\right.+\left.\gamma \max\nolimits_{\tilde{a}} Q_{\mathcal{L}}(s^{\prime},\tilde{a})\right],
\end{equation}
where the step size $\beta_t(s,a)$ can be chosen as (\textit{cf}. \cite{Q-rate})
\[
\beta_t(s,a)=
\begin{cases}
|\#(s,a,t)+1|^{-h}, & (s,a)=(s_t,a_t),\\
0, & \text{otherwise}.
\end{cases}
\]
with $h\in(1/2,1)$. Here $\#(s,a,t)$ is the number of times up to time $t$ that one visits the  pair $(s,a)$.  
The algorithm then proceeds to choose action $a'$ based on $Q_{\mathcal{L}}$ with appropriate exploration strategies, including the $\epsilon$-greedy strategy. 

After obtaining the approximate $\hat{Q}_{\mathcal{L}}^\star$, in order to retrieve an approximately optimal policy, 
 it would be natural to define an \textbf{argmax-e} operator  so that actions with equal maximum Q-values would have equal probabilities to be selected. 
Unfortunately, the discontinuity and sensitivity of  \textbf{argmax-e}  could lead to an unstable algorithm (see Figure \ref{fig:naive} for the corresponding naive Algorithm \ref{AQL_MFG} in Appendix).  
 \footnote{\textbf{argmax-e} is not continuous:
Let $x=(1,1)$, then $\textbf{argmax-e}(x)=(1/2,1/2)$.
For any $\epsilon>0$, let $y=(1,1-\epsilon)$, then $\textbf{argmax-e}(y)=(1,0)$.}

Instead, 
we consider a Boltzmann policy based on the operator $\textbf{softmax}_c:\mathbb{R}^n\rightarrow\mathbb{R}^n$, defined as 
\begin{eqnarray}\label{eqn:soft_max_operator}
\textbf{softmax}_c(x)_i=\frac{\exp(cx_i)}{\sum_{j=1}^n\exp(cx_j)}.
\end{eqnarray}
This operator is smooth and  close to the \textbf{argmax-e} (see Lemma \ref{soft-arg-diff} in the Appendix). Moreover, 
even though Boltzmann policies are not optimal, the difference between the Boltzmann and the optimal one can always be controlled by choosing the hyper-parameter $c$ appropriately in the \textbf{softmax} operator. {\color{black}Note that other smoothing operators (\textit{e.g.}, Mellowmax \cite{Mellowmax}) may also be considered in the future.}

\paragraph{Step 2: error control in updating $\mathcal {L}$.}
Given the sub-optimality of the  Boltzmann policy, one needs to characterize the difference between the optimal policy and the non-optimal ones. In particular, one can define the action gap  between the best action and the second best action in terms of the  Q-value as $\delta^s(
\mathcal{L}):=\max_{a'\in\mathcal{A}}Q^\star_{\mathcal{L}}(s,a')-\max_{a\notin \text{argmax}_{a\in\mathcal{A}}Q^\star_{\mathcal{L}}(s,a)}Q^\star_{\mathcal{L}}(s,a)>0$. 
Action gap is important  for approximation algorithms \cite{gap-increase}, and are closely related to the problem-dependent bounds for regret analysis in reinforcement learning and  multi-armed bandits, and advantage learning algorithms including A2C \cite{A2C}. 

The problem is: in order for the learning algorithm to converge in terms of $\mathcal{L}$ (Theorem \ref{conv_AIQL}), one needs to ensure a definite differentiation between the optimal policy and the sub-optimal ones. This is problematic as the infimum of $\delta^s(\mathcal{L})$ over an infinite number of $\mathcal{L}$ can be $0$. To address this, the population distribution at step $k$, say  $\mathcal{L}_k$,
 needs to be projected to a finite grid,  called $\epsilon$-net. The relation between
the $\epsilon$-net and \text{action gaps} is as follows:

\noindent\textit{For any $\epsilon>0$,
there exist a positive function $\phi(\epsilon)$ and  an $\epsilon$-net $S_{\epsilon}:=\{\mathcal{L}^{(1)},\dots,\mathcal{L}^{(N_{\epsilon})}\}$ $\subseteq$ $\mathcal{P}(\mathcal{S}\times\mathcal{A})$, with the properties that $\min_{i=1,\dots,N_{\epsilon}} d_{TV}(\mathcal{L},\mathcal{L}^{(i)})\leq \epsilon$ for any $\mathcal{L}\in\mathcal{P}(\mathcal{S}\times\mathcal{A})$, and that 
$\max_{a'\in\mathcal{A}}Q^\star_{\mathcal{L}^{(i)}}(s,a')-Q^\star_{\mathcal{L}^{(i)}}(s,a)\geq \phi(\epsilon)$
 for any $i=1,\dots,N_{\epsilon}$,  $s\in\mathcal{S}$, and any $a\notin \text{argmax}_{a\in\mathcal{A}}Q^\star_{\mathcal{L}^{(i)}}(s,a)$. 
}

Here the existence of $\epsilon$-nets is trivial due to the compactness of the probability simplex $\mathcal{P}(\mathcal{S}\times\mathcal{A})$, and the existence of $\phi(\epsilon)$ comes from the finiteness of the action set $\mathcal{A}$.
In practice, $\phi(\epsilon)$ often takes the form of  $D\epsilon^{\alpha}$  with $D>0$ and the exponent $\alpha>0$ characterizing the decay rate  of the action gaps. 

Finally,  to enable Q-learning, it is assumed that one has access to a population simulator (See \cite{Batch_MARL, MARL_PDO}). That is,
for any policy $\pi\in\Pi$, given the current state $s\in\mathcal{S}$, for any population distribution $\mathcal{L}$, one can obtain the next state $s'\sim P(\cdot|s,\pi(s,\mu),\mathcal{L})$, a reward $r= r(s,\pi(s,\mu),\mathcal{L})$, and the next population distribution $\mathcal{L}'=\mathbb{P}_{s',\pi(s',\mu)}$. For brevity, we denote the simulator as $(s',r, \mathcal{L}')=\mathcal{G}(s,\pi,\mathcal{L})$. Here $\mu$ is the state marginal distribution of $\mathcal{L}$. 

In summary, we propose the following Algorithm \ref{AIQL_MFG}. 


\begin{algorithm}[h]
  \caption{\textbf{Q-learning for GMFGs (GMF-Q)}}
  \label{AIQL_MFG}
\begin{algorithmic}[1]
  \STATE \textbf{Input}: Initial $\mathcal{L}_0$, tolerance $\epsilon>0$.
 \FOR {$k=0, 1, \cdots$}
  \STATE Perform Q-learning for $T_k$ iterations to find the approximate Q-function $\hat{Q}_k^\star(s,a)=\hat{Q}^\star_{\mathcal{L}_k}(s,a)$ of an MDP with dynamics $P_{\mathcal{L}_k}(s'|s,a)$ and rewards $r_{\mathcal{L}_k}(s,a)$.
  \STATE Compute $\pi_k\in\Pi$ with $\pi_k(s)=\textbf{softmax}_c(\hat{Q}^\star_k(s,\cdot))$.
  \STATE Sample $s\sim \mu_k$ ($\mu_k$ is the population state marginal of $\mathcal{L}_k$), obtain $\tilde{\mathcal{L}}_{k+1}$ from $\mathcal{G}(s,\pi_k,\mathcal{L}_k)$.
  \STATE Find $\mathcal{L}_{k+1}=\textbf{Proj}_{S_{\epsilon}}(\tilde{\mathcal{L}}_{k+1})$
\ENDFOR
\end{algorithmic}
\end{algorithm}

Note that \textbf{softmax} is applied only at the end of each outer iteration when a good approximation of $Q$ function is obtained. Within the outer iteration for the MDP problem with fixed mean-field information, standard Q-learning method is applied.

Here $\textbf{Proj}_{S_{\epsilon}}(\mathcal{L})=\text{argmin}_{\mathcal{L}^{(1)},\dots,\mathcal{L}^{(N_{\epsilon})}}d_{TV}(\mathcal{L}^{(i)},\mathcal{L})$. 
{\color{black}For computational tractability, it would be sufficient to choose $S_{\epsilon}$ as a truncation grid so that projection of $\tilde{\mathcal{L}}_k$ onto the epsilon-net reduces to truncating $\tilde{\mathcal{L}}_k$ to a certain number of digits. For instance, in our experiment, the number of digits is chosen to be 4. The choices of the hyper-parameters $c$ and $T_k$ can be found in Lemma \ref{Q-finite-bd} and Theorem \ref{conv_AIQL}. In practice, the algorithm is rather robust  with respect to these hyper-parameters.} 

In the  special case when the rewards $r_{\mathcal{L}}$ and transition dynamics $P(\cdot|s,a,\mathcal{L})$ are known, one can replace the Q-learning step in the above Algorithm \ref{AIQL_MFG}  by a value iteration, resulting in the GMF-V Algorithm \ref{VI} in the Appendix.

We next show the convergence of this GMF-Q algorithm (Algorithm \ref{AIQL_MFG}) to an $\epsilon$-Nash of \eqref{mfg}, with complexity analysis.

\begin{theorem}[Convergence and complexity of GMF-Q]\label{conv_AIQL}
Assume the same conditions in {\color{black}Theorem \ref{thm1_stat}} and  Lemma \ref{Q-finite-bd} in the Appendix. For any tolerances $\epsilon,~\delta>0$, set $\delta_k=\delta/ K_{\epsilon,\eta}$, $\epsilon_k=(k+1)^{-(1+\eta)}$ for some $\eta\in(0,1]$ $(k=0,\dots,K_{\epsilon,\eta}-1)$, 
    $T_k=T^{\mathcal{M}_{\mathcal{L}_k}}(\delta_k,\epsilon_k)$ (defined in Lemma \ref{Q-finite-bd} in the Appendix) and  $c=\frac{\log(1/\epsilon)}{\phi(\epsilon)}$. Then with probability at least $1-2\delta$, $W_1(\mathcal{L}_{K_{\epsilon,\eta}},\mathcal{L}^\star)\leq C\epsilon$.   \\
    Moreover,  the total number of iterations $T=\sum_{k=0}^{K_{\epsilon,\eta} -1}T^{\mathcal{M}_{\mathcal{L}_k}}(\delta_k,\epsilon_k)$ is bounded by \footnote{{\color{black}Let $h=\frac{3}{4}$, $\eta=1$, the bound reduces to $T=O(K_{\epsilon}^{\frac{19}{3}}(\log(\frac{K_{\epsilon}}{\delta}))^{\frac{41}{3}})$. Note that this bound may not be tight.}}
\begin{equation}\label{Tbound}
T=O\left(K_{\epsilon,\eta}^{1+\frac{4}{h}}\left(\log(K_{\epsilon,\eta}/\delta)\right)^{\frac{2}{1-h}+\frac{2}{h}+3}\right).
\end{equation}
Here $K_{\epsilon,\eta}:=\left\lceil 2\max\left\{(\eta\epsilon{\color{black}/c})^{-1/\eta},\log_d(\epsilon/{\color{black}\max\{\text{diam}(\mathcal{S})\text{diam}(\mathcal{A}),{\color{black}c}\}})+1\right\}\right\rceil$ is the number of outer iterations, $h$ is the step-size exponent in Q-learning (defined in Lemma \ref{Q-finite-bd} in the Appendix), and the constant $C$ is independent of $\delta$, $\epsilon$ and $\eta$. 
\end{theorem}

The proof of Theorem \ref{conv_AIQL} in the Appendix depends on the Lipschitz continuity of the \textbf{softmax} operator \cite{softmax}, the closeness between \textbf{softmax} and the \textbf{argmax-e} (Lemma~\ref{soft-arg-diff} in the Appendix), and the complexity of Q-learning for the MDP (Lemma~\ref{Q-finite-bd} in the Appendix). 

\section{Experiment: repeated auction game}\label{experiments}

In this section, we report the performance of the proposed GMF-Q Algorithm.
The objectives of the experiments include 1) testing the convergence, stability, and learning ability of GMF-Q  in the GMFG setting, and 2) comparing GMF-Q with existing multi-agent reinforcement learning algorithms, including IL algorithm and MF-Q algorithm.


 We take  the GMFG framework for  the repeated auction game from Section \ref{section:example}. Here each advertiser learns to bid in the auction with a budget constraint.

 \paragraph{Parameters.}
The model parameters are set as: $|\mathcal{S}|=|\mathcal{A}|=10$,  
the overbidding penalty  $\rho=0.2$,  
 the distributions of the conversion rate $v \sim$ uniform(${\color{black}\{1,2,3,4\}})$,
 and the competition intensity index $M=5$.
 The random fulfillment
is chosen as:
if $s < s_{\max}$,  $\Delta(s)=1$ with probability $\frac{1}{2}$ and $\Delta(s)=0$ with probability $\frac{1}{2}$; if $s=s_{\max}$, $\Delta(s)=0$.

The algorithm parameters are (unless otherwise specified): the temperature parameter  $c=4.0$, the discount factor  $\gamma=0.8$,   the parameter $h$ from Lemma \ref{Q-finite-bd} in the Appendix being $h=0.87$, and the baseline inner iteration being $2000$.  
{Recall that for  GMF-Q, both   $v$ and the dynamics of $P$ for $s$ are unknown {\it a priori}.}
{\color{black}The $90\%$-confidence intervals are calculated with $20$ sample paths.}
\paragraph{Performance evaluation in the GMFG setting.}
Our experiment shows that the GMF-Q Algorithm is efficient and robust, and learns well. 
\paragraph{{\it Convergence and stability of GMF-Q.}}
GMF-Q  is efficient and robust. First, GMF-Q converges after about $10$ outer iterations; secondly, as the number of inner iterations increases, the error decreases (Figure~\ref{fig:inexact}); 
and finally,  the convergence is robust with respect to both the change of number of states and  the initial population distribution (Figure~\ref{fig:different_state}).
  
 In contrast, the  Naive algorithm does not converge even with $10000$ inner iterations,  and the {\color{black}joint distribution $\mathcal{L}_t$} keeps fluctuating (Figure~\ref{fig:naive}). 

\paragraph{{\it Learning accuracy of GMF-Q.}}
GMF-Q learns well. Its learning accuracy is tested against its special form GMF-V {\color{black}(Appendix \ref{GMF-V})}, with the latter assuming a known distribution of conversion rate $v$ and the dynamics $P$ for the budget $s$.
The relative $L_2$ distance between the Q-tables of these two algorithms is
$\Delta Q := \frac{\|Q_{\text{GMF-V}}-Q_{\text{GMF-Q}}\|_2}{\|Q_{\text{GMF-V}}\|_2}=0.098879$. This implies that GMF-Q  learns the true GMFG solution with $90$-percent accuracy  with $10000$ inner iterations.   

The heatmap in Figure~\ref{fig:q_learn_table} is the Q-table for GMF-Q Algorithm after $20$ outer iterations. Within each outer iteration, there are $T_k^{\text{GMF-Q}} = 10000$ inner iterations. The heatmap in Figure~\ref{fig:q_iteration_table} is the Q-table for GMF-Q Algorithm after $20$ outer iterations. Within each outer iteration, there are $T_k^{\text{GMF-V}} = 5000$ inner iterations. 

\begin{table}
		\begin{center}
    \caption{Q-table with $T_k^{\text{GMF-V}} = 5000$.}
    \label{tab:q_comparison}
    \begin{tabular}{c|c|c|c|c} 
          $T_k^{\text{GMF-Q}} $ &1000& 3000 & 5000&10000 \\
      \hline
      $\Delta Q$ &0.21263 & 0.1294& 0.10258&0.0989
    \end{tabular}
  \end{center}
  \end{table}
\begin{figure}
		\centering
\subfigure[\text{GMF-Q}.]{ \label{fig:q_learn_table}
\includegraphics[width=.39\textwidth]{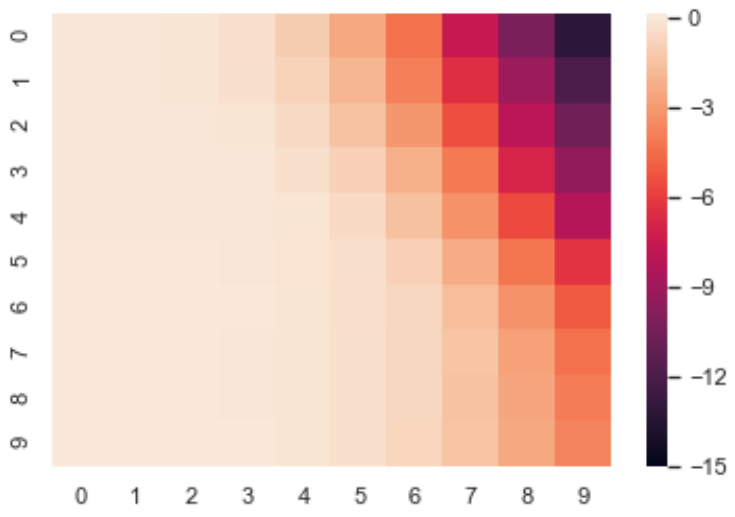}
}
\subfigure[GMF-V.]{\label{fig:q_iteration_table}
\includegraphics[width=.39\textwidth]{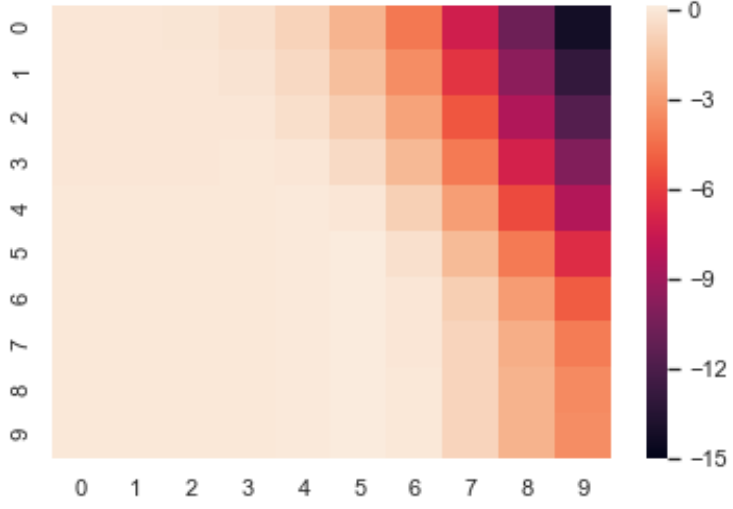}}
\caption{Q-tables: GMF-Q vs. GMF-V.}\label{fig:q_tables}
\end{figure}

\begin{figure}
  \centering
  \begin{minipage}[b]{0.48\textwidth}
    \centering
\includegraphics[width=.97\textwidth]{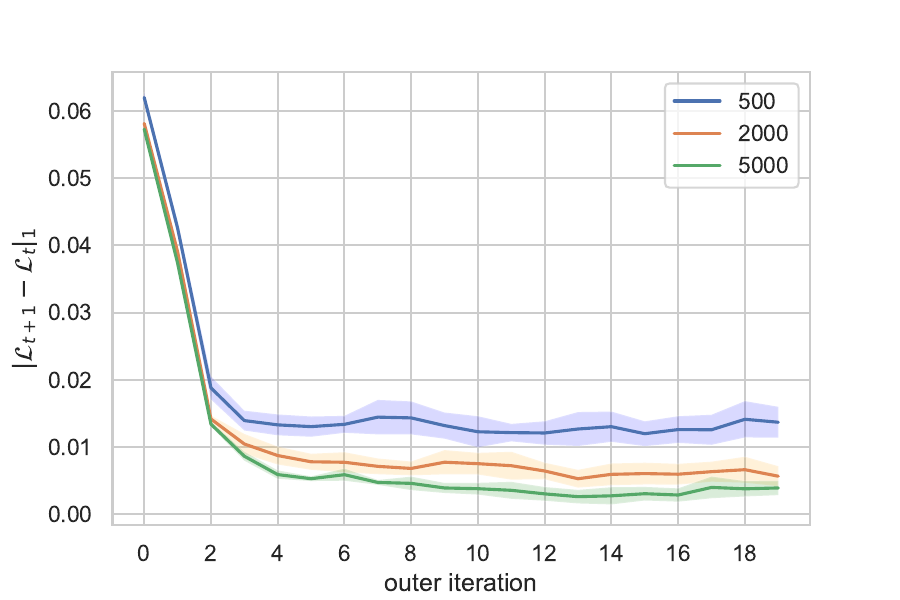}
\caption{Convergence with different  \\\hspace{\textwidth} number of inner iterations.}
\label{fig:inexact} 
  \end{minipage}
  \hfill
  \begin{minipage}[b]{0.48\textwidth}
    \centering
\includegraphics[width=.97\textwidth]{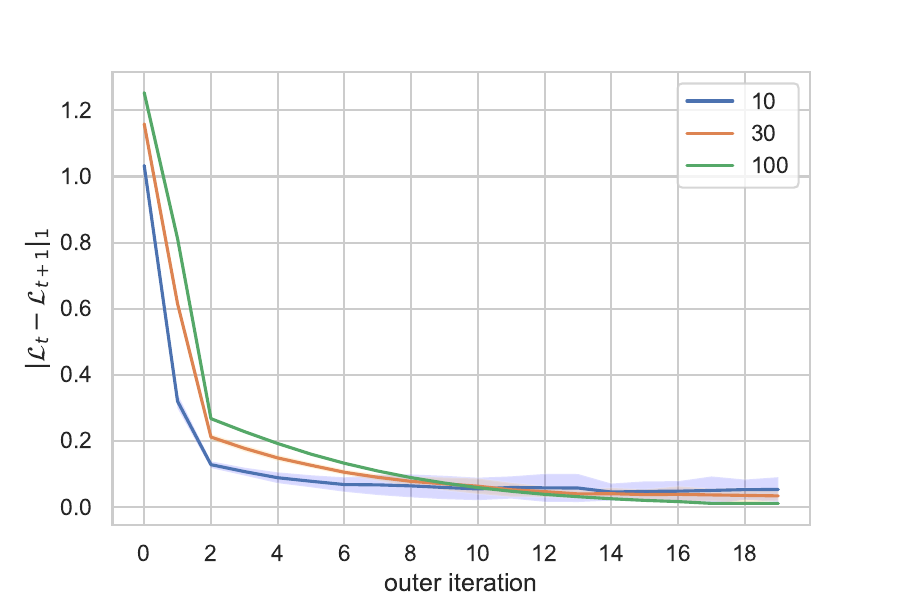}
\caption{Convergence with  different \\\hspace{\textwidth}number of states.}\label{fig:different_state}
  \end{minipage}
  \end{figure}
 
 \begin{figure} 
    \centering
\subfigure[fluctuation in $l_{\infty}$.]{ \label{fig:naive_1}
\includegraphics[width=.42\textwidth]{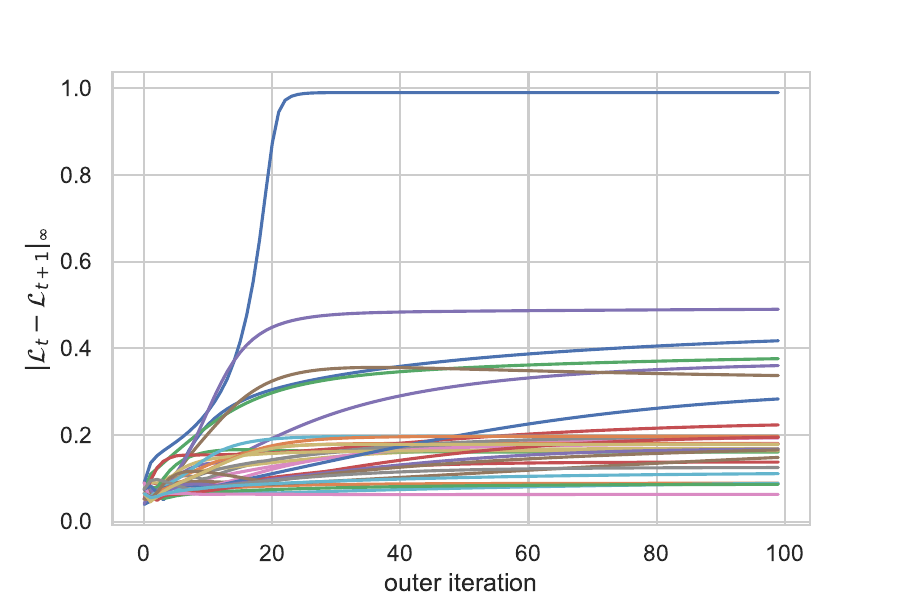}
}
\subfigure[fluctuation in $l_{1}$.]{ \label{fig:naive_1}
\includegraphics[width=.42\textwidth]{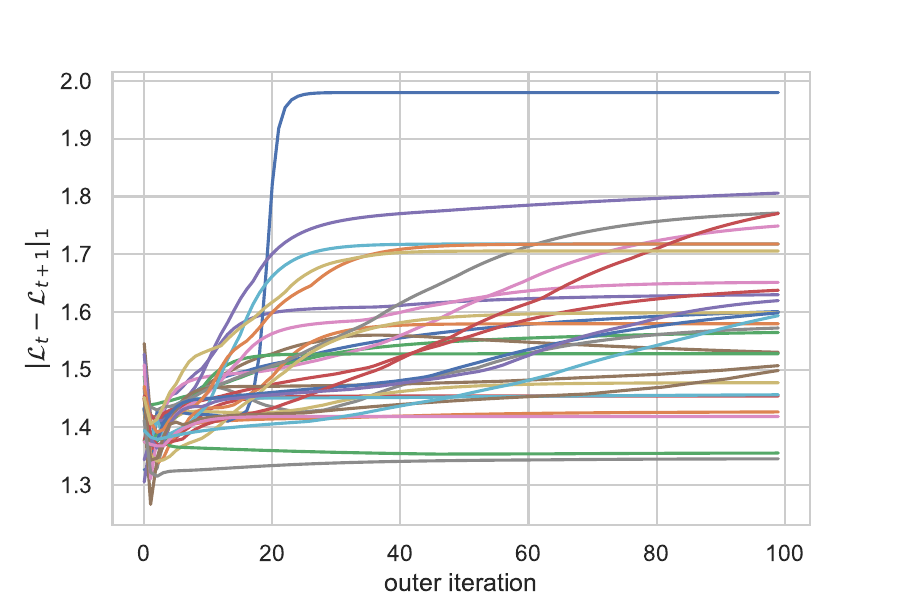}
}
\caption{Fluctuations of Naive Algorithm (30 sample paths).}\label{fig:naive}
\end{figure}

\begin{figure}
\centering
\subfigure[$|\mathcal{S}|=|\mathcal{A}|=10,{\color{black}N}=20$.]{ \label{fig:comparison1}
\includegraphics[width=.3\textwidth]{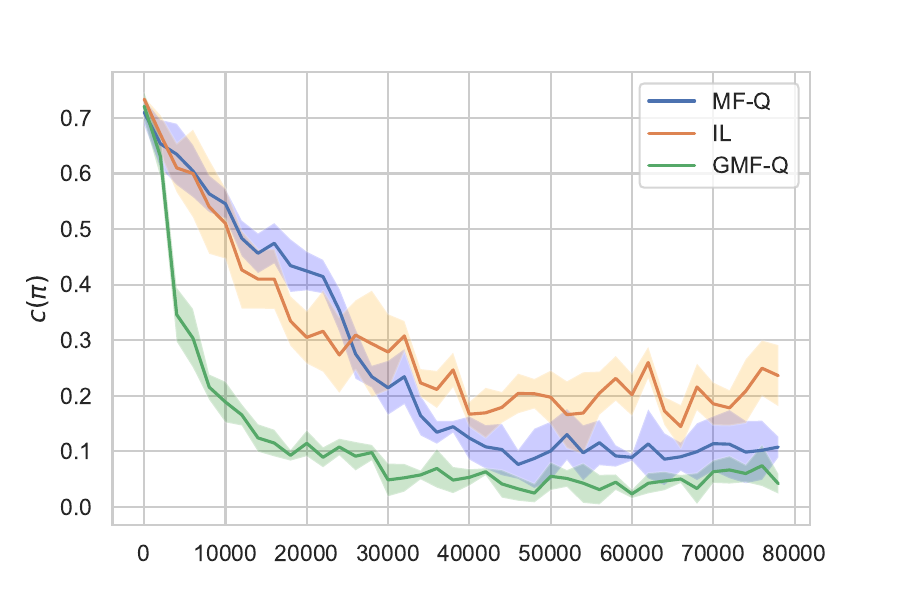}
}
\subfigure[$|\mathcal{S}|=|\mathcal{A}|=20,{\color{black}N}=20$.]{ \label{fig:comparison1}
\includegraphics[width=.3\textwidth]{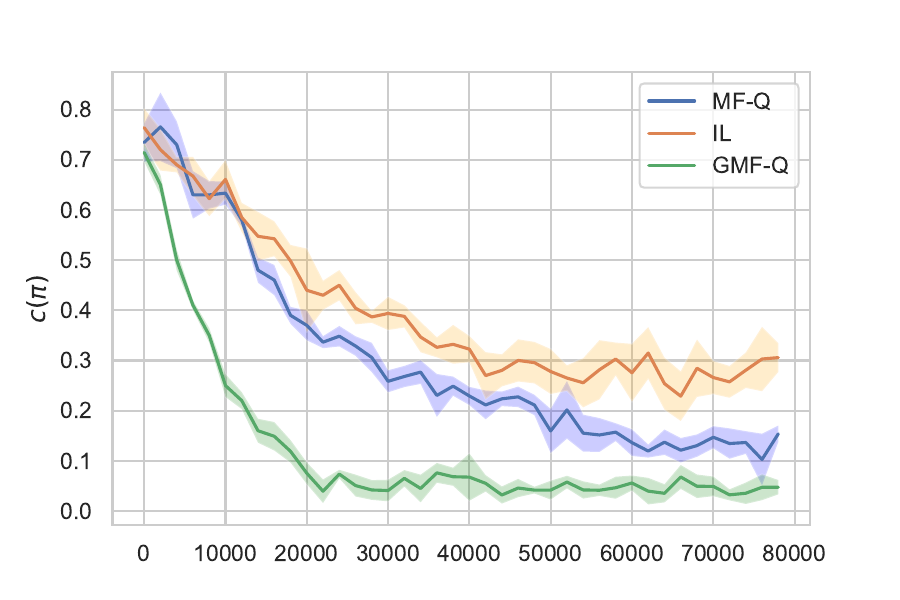}
}
\subfigure[$|\mathcal{S}|=|\mathcal{A}|=10,{\color{black}N}=40$.]{\label{fig:comparison2}
\includegraphics[width=.3\textwidth]{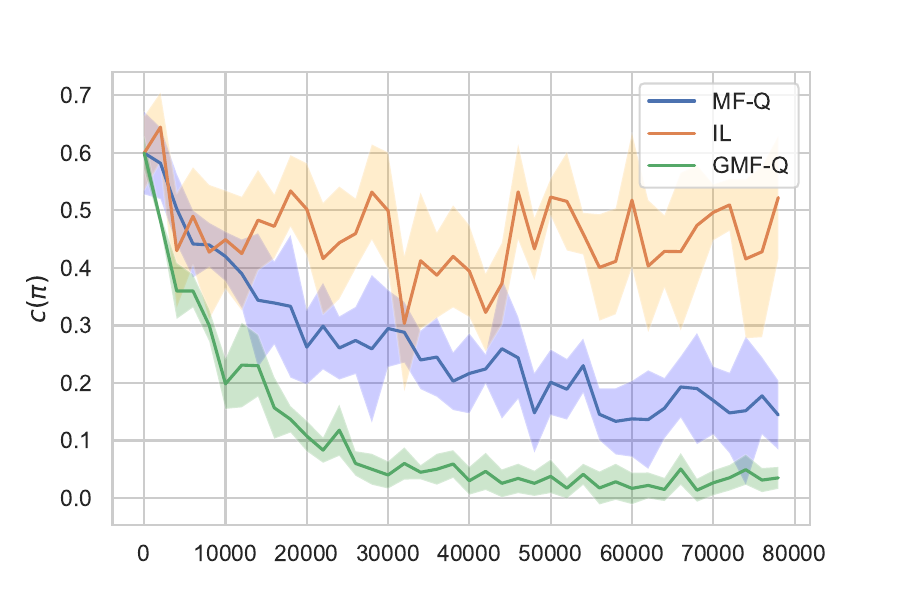}}

\caption{Learning accuracy based on $C(\pmb{\pi})$.}
\label{fig:comparison}
\end{figure}

\paragraph{Comparison with existing algorithms for $N$-player games.}
To test the effectiveness of GMF-Q for approximating $N$-player games, we next compare GMF-Q with  IL algorithm and  MF-Q algorithm. IL algorithm \cite{T1993} considers $N$ independent players and each player solves a decentralized reinforcement learning problem ignoring other players in the system. The MF-Q algorithm \cite{YLLZZW2018} extends the NASH-Q Learning algorithm for the $N$-player game introduced in \cite{HW2003}, adds the aggregate actions $(\bar{\pmb{a}}_{-i}=\frac{\sum_{j \ne i}a_j}{{\color{black}N}-1})$ from the opponents, and  works for the class of games where the interactions are only through the average actions of $N$ players.

 \paragraph{{\it Performance metric.}} 
We adopt  the following metric  to measure the difference between  a given policy $\pi$ and an NE (here $\epsilon_0>0$ is a safeguard, and is taken as $0.1$ in the experiments):
$$C(\pmb{\pi}) = \frac{1}{N|\mathcal{S}|^N}\sum\nolimits_{i=1}^N \sum\nolimits_{\pmb{s} \in \mathcal{S}^N  }\dfrac{\max_{{\pi}^i} V_i(\pmb{s},({\pmb{\pi}^{-i}},\pi^i))-V_i(\pmb{s},{\pmb{\pi}})}{|\max_{{\pi}^i} V_i(\pmb{s},({\pmb{\pi}^{-i}},\pi^i))|+\epsilon_0}.$$
Clearly $C(\pmb{\pi}) \geq 0$, and $C(\pmb{\pi}^*)=0$ if and only if $\pmb{\pi}^*$ is an NE. 
 Policy $\arg \max _{{\pi}_i} V_i(\pmb{s},({\pmb{\pi}^{-i}},\pi_i))$ is called the best response to $\pmb{\pi}^{-i}$.  A similar metric without normalization has been adopted in \cite{PPP2018}.

Our experiment (Figure \ref{fig:comparison}) shows that  GMF-Q is superior in terms of convergence rate, accuracy, and stability for approximating an $N$-player game:
GMF-Q  converges faster than  IL and  MF-Q, with the smallest error, and with the lowest variance, as {\color{black}$\epsilon$-net improves the stability.}

For instance, when $N=20$,  IL Algorithm converges  with the largest error $0.220$. 
The error from MF-Q is $0.101$,  smaller than  IL  but  still bigger than the error from GMF-Q.
The GMF-Q  converges with  the lowest error $0.065$. {\color{black} Moreover, as $N$ increases, the error of GMF-Q deceases while the errors of both MF-Q and IL increase significantly. As $|\mathcal{S}|$ and  $|\mathcal{A}|$ increase, GMF-Q is robust with respect to this increase of dimensionality, while 
 both MF-Q and IL clearly suffer from the increase of the dimensionality with decreased  convergence rate and accuracy.}
Therefore, GMF-Q is more scalable than IL and MF-Q, when the system is complex and the number of players $N$ is large.

\section{Conclusion}
This paper builds a GMFG framework for simultaneous learning and decision-making, establishes the existence and uniqueness of NE, and proposes a Q-learning algorithm GMF-Q with convergence and complexity analysis. Experiments demonstrate superior performance of GMF-Q.

\section*{Acknowledgment}
We thank Haoran Tang for the insightful early discussion on stabilizing the Q-learning algorithm and sharing the ideas of his work on soft-Q-learning \cite{softQ}, which motivates our adoption of the soft-max operators. {\color{black}We also thank the anonymous NeurIPS 2019 reviewers for the valuable suggestions.}

\newpage
\bibliography{mfg_rl.bib}
\bibliographystyle{plain}

\newpage
\appendix

\section{Distance metrics and completeness} 
This section reviews some basic properties of the Wasserstein distance. 
It then proves that the metrics defined in the main text are indeed distance functions and define complete metric spaces. 

\paragraph{$\ell_1$-Wasserstein distance and dual representation.} 
The $\ell_1$ Wasserstein distance over $\mathcal{P}(\mathcal{X})$ for $\mathcal{X}\subseteq\mathbb{R}^k$ is defined as 
\begin{equation}\label{W1}
W_1(\nu,\nu'):=\inf_{M\in\mathcal{M}(\nu,\nu')}\int_{\mathcal{X}\times\mathcal{X}}\|x-y\|_2\text{d}M(x,y).
\end{equation}
where $\mathcal{M}(\nu,\nu')$ is the set of all measures (couplings) on $\mathcal{X}\times\mathcal{X}$, with marginals $\nu$ and $\nu'$ on the two components, respectively.

The Kantorovich duality theorem enables the following equivalent dual representation of $W_1$:
\begin{equation}\label{W1_dual}
W_1(\nu,\nu')=\sup_{\|f\|_L\leq 1}\left|\int_{\mathcal{X}}fd\nu-\int_{\mathcal{X}}fd\nu'\right|,
\end{equation}
where the supremum is taken over all $1$-Lipschitz functions $f$, \textit{i.e.}, $f$ satisfying $|f(x)-f(y)|\leq \|x-y\|_2$ for all $x,y\in\mathcal{X}$.

The Wasserstein distance $W_1$ can also be related to the total variation distance via the following inequalities \cite{metrics_prob}: 
\begin{equation}\label{W1_tv}
d_{\min}(\mathcal{X})d_{TV}(\nu,\nu')\leq W_1(\nu,\nu')\leq \text{diam}(\mathcal{X})d_{TV}(\nu,\nu'),
\end{equation}
where $d_{\min}(\mathcal{X})=\min_{x\neq y\in\mathcal{X}}\|x-y\|_2$, which is guaranteed to be positive when $\mathcal{X}$ is finite.

When $\mathcal{S}$ and $\mathcal{A}$ are compact,  for any compact subset $\mathcal{X}\subseteq\mathbb{R}^k$, and for any $\nu,\nu'\in\mathcal{P}(\mathcal{X})$, $W_1(\nu,\nu')\leq \text{diam}(\mathcal{X})d_{TV}(\nu,\nu')\leq \text{diam}(\mathcal{X})<\infty$, where $\text{diam}(\mathcal{X})=\sup_{x,y\in\mathcal{X}}\|x-y\|_2$ and $d_{TV}$ is the total variation distance.
Moreover,  one can verify  
\begin{lemma}\label{metrics}
Both $D$ and $\mathcal{W}_1$ are distance functions, and they are finite for any input distribution pairs. In addition, both $(\{\Pi\}_{t=0}^{\infty},D)$ and $(\{\mathcal{P}(\mathcal{S}\times\mathcal{A})\}_{t=0}^{\infty}, \mathcal{W}_1)$ are \textit{complete metric spaces}.
\end{lemma}
These facts enable the usage of Banach fixed-point mapping theorem for the proof of existence and uniqueness (Theorems \ref{thm1} and \ref{thm1_stat}).
\begin{proof}[Proof of Lemma \ref{metrics}]
It is known that for any compact set $\mathcal{X}\subseteq\mathbb{R}^k$, $(\mathcal{P}(\mathcal{X}), W_1)$ defines a complete metric space \cite{wass_comp}. Since $W_1(\nu,\nu')\leq \text{diam}(\mathcal{X})$ is uniformly bounded for any $\nu,~\nu'\in\mathcal{P}(\mathcal{X})$, we know that $\mathcal{W}_1(\pmb{\mathcal{L}},\pmb{\mathcal{L}}')\leq \text{diam}(\mathcal{X})$ and $D(\pmb{\pi},\pmb{\pi'})\leq \text{diam}(\mathcal{X})$ as well, so they are both finite for any input distribution pairs. It is clear that they are distance functions based on the fact that $W_1$ is a distance function.

Finally, we show the completeness of the two metric spaces $(\{\Pi\}_{t=0}^{\infty},D)$ and $(\{\mathcal{P}(\mathcal{S}\times\mathcal{A})\}_{t=0}^{\infty}, \mathcal{W}_1)$. Take $(\{\Pi\}_{t=0}^{\infty},D)$ for example. Suppose that $\pmb{\pi}^k$ is a Cauchy sequence in $(\{\Pi\}_{t=0}^{\infty},D)$. Then for any $\epsilon>0$, there exists a positive integer $N$, such that for any $m,~n\geq N$, 
\begin{equation}
D(\pmb{\pi}^n,\pmb{\pi}^{m})\leq \epsilon\Longrightarrow W_1(\pi_t^n(s),\pi_t^{m}(s))\leq \epsilon\text{ for any $s\in\mathcal{S}$, $t\in\mathbb{N}$},
\end{equation}
which implies that $\pi_t^k(s)$ forms a Cauchy sequence in $(\mathcal{P}(\mathcal{A}),W_1)$, and hence by the completeness of $(\mathcal{P}(\mathcal{A}),W_1)$, $\pi_t^k(s)$ converges to some $\pi_t(s)\in\mathcal{P}(\mathcal{A})$. As a result,  $\pmb{\pi}^n\rightarrow\pmb{\pi}\in\{\Pi\}_{t=0}^{\infty}$ under metric $D$, which shows that $(\{\Pi\}_{t=0}^{\infty},D)$ is complete. 

The completeness of $(\{\mathcal{P}(\mathcal{S}\times\mathcal{A})\}_{t=0}^{\infty}, \mathcal{W}_1)$ can be proved similarly.
\end{proof}

The same argument for Lemma \ref{metrics} shows that   both $D$ and $W_1$ are distance functions and are finite for any input distribution pairs, with both $(\Pi, D)$ and $(\mathcal{P}(\mathcal{S}\times\mathcal{A}),W_1)$  again complete metric spaces.

\section{Existence and uniqueness for stationary NE of GMFGs} \label{stat_mfg_app}
\begin{definition}[Stationary NE for GMFGs]\label{nash2_stat} 
In \eqref{mfg}, a player-population profile ($\pi^\star$, $\mathcal{L}^\star$)
  is called a stationary NE if 
\begin{enumerate}
    \item (Single player side) For any policy $\pi$ and any initial state $s\in \mathcal{S}$, 
\begin{equation}
V\left(s,\pi^\star,{\color{black}\mathcal{L}^\star}\right)\geq V\left(s,\pi,{\color{black}\mathcal{L}^\star}\right).
\end{equation}
\item  (Population side) $\mathbb{P}_{s_t,a_t}= {\mathcal{L}^{\star}}$ for all $t\geq 0$, where $\{s_t,a_t\}_{t=0}^{\infty}$ is the dynamics under the policy  $\pi^\star$ starting from $s_0 \sim \mu^{\star}$, with $a_t\sim\pi^\star(s_t,{\color{black}\mu^{\star}})$, $s_{t+1}\sim P(\cdot|s_t,a_t,{\color{black}\mathcal{L}^\star})$, and $\mu^{\star}$ being the population state marginal of $\mathcal{L}^\star$.
\end{enumerate}
\end{definition}

The existence and uniqueness of the NE  to (\ref{mfg}) in the stationary setting can be established by modifying  appropriately the same  fixed-point approach for the GMFG in the main text. 

\paragraph{Step 1.}
Fix $\mathcal{L}$, the GMFG becomes the  classical optimization problem.  That is,  solving (\ref{mfg}) is now reduced to finding a policy
 $\pi_{\mathcal{L}}^\star\in \Pi:=\{\pi\,|\,\pi:\mathcal{S}\rightarrow\mathcal{P}(\mathcal{A})\}$ to maximize 
 \[
\begin{array}{ll}
  V(s,\pi_{\mathcal{L}}, \mathcal{L}):= &\mathbb{E}\left[\sum\limits_{t=0}^{\infty}\gamma^tr(s_t,a_t,\mathcal{L})|s_0=s\right],\\
\text{subject to}& s_{t+1}\sim P(s_t,a_t,\mathcal{L}),\quad a_t\sim\pi_{\mathcal{L}}(s_t).
\end{array}
\]
Now given this fixed $\mathcal{L}$ and the solution $\pi_{\mathcal{L}}^\star$ to the above optimization problem, one can again
define 
$$\Gamma_1:\mathcal{P}(\mathcal{S}\times \mathcal{A})\rightarrow\Pi, $$
such that  $\pi_{\mathcal{L}}^\star=\Gamma_1(\mathcal{L})$.
Note that this $\pi_{\mathcal{L}}^\star$ satisfies the single player side condition for the population state-action pair $L$, 
\begin{equation}
V\left(s,\pi_{\mathcal{L}}^\star,\mathcal{L}\right)\geq V\left(s,\pi,\mathcal{L}\right),
\end{equation}
for any policy $\pi$ and any initial state $s\in\mathcal{S}$. 

Accordingly, a similar feedback regularity condition is needed in this step.
\begin{assumption}\label{policy_assumption_stat}
There exists a constant $d_1\geq 0$,
 such that for any $\mathcal{L}, \mathcal{L}' \in \mathcal{P}(\mathcal{S}\times\mathcal{A})$, 
\begin{equation}\label{Gamma1_lip}
D(\Gamma_1(\mathcal{L}),\Gamma_1(\mathcal{L}'){\color{black})} \leq d_1W_1(\mathcal{L}, \mathcal{L}'),
\end{equation}
where 
\begin{equation}
\begin{split}
D(\pi,\pi')&:=\sup_{s\in\mathcal{S}}W_1(\pi(s),\pi'(s)),
\end{split}
\end{equation}
and $W_1$ is the $\ell_1$-Wasserstein distance (a.k.a. earth mover distance) between probability measures. 
\end{assumption}

 \paragraph{Step 2.} Based on the analysis of Step 1 and $\pi_{\mathcal{L}}^\star$, update the initial $\mathcal{L}$ to $\mathcal{L}'$ following the controlled dynamics $P(\cdot|s_t, a_t,\mathcal{L})$.

Accordingly, 
define a mapping $\Gamma_2:\Pi\times \mathcal{P}(\mathcal{S}\times\mathcal{A})\rightarrow \mathcal{P}(\mathcal{S}\times\mathcal{A})$ as follows: 
\begin{eqnarray}
\Gamma_2(\pi, \mathcal{L}):=\hat{\mathcal{L}}= \mathbb{P}_{s_1,a_1},
\end{eqnarray}
where {\color{black}$a_1\sim \pi(s_1)$}, $s_{1}\sim \mu P(\cdot|\cdot,a_0,\mathcal{L})$,  $a_0\sim\pi(s_0)$, $s_0 \sim \mu$, and $\mu$ is the population state marginal of $\mathcal{L}$.

One also needs a similar assumption in this step.
\begin{assumption}\label{population_assumption_stat}
There exist constants $d_2,~d_3\geq 0$, such that for any admissible policies $\pi,\pi_1,\pi_2$ and joint distributions $\mathcal{L}, \mathcal{L}_1, \mathcal{L}_2$, 
\begin{equation}\label{Gamma2_lip1_stat}
W_1(\Gamma_2(\pi_1,\mathcal{L}),\Gamma_2(\pi_2,\mathcal{L})) \leq d_2 D(\pi_1,\pi_2),\\
\end{equation}
\begin{equation}\label{Gamma2_lip2_stat}
W_1(\Gamma_2(\pi,\mathcal{L}_1),\Gamma_2(\pi,\mathcal{L}_2)) \leq d_3 W_1(\mathcal{L}_1,\mathcal{L}_2).
\end{equation}
\end{assumption}

\paragraph{Step 3.} Repeat until $\mathcal{L}'$ matches $\mathcal{L}$.

This step is to ensure the population side condition. To ensure the convergence of
 the combined step one and step two,  it suffices if  $\Gamma:\mathcal{P}(\mathcal{S}\times\mathcal{A})\rightarrow \mathcal{P}(\mathcal{S}\times\mathcal{A})$ with $\Gamma(\mathcal{L}):=\Gamma_2(\Gamma_1(\mathcal{L}), \mathcal{L})$ is a contractive mapping (under  the $W_1$ distance). 
 
Similar to the proof of Theorem \ref{thm1},  again by the Banach fixed point theorem and the completeness of the related metric spaces, there exists  a unique stationary NE of the GMFG.  That is, 

 \begin{theorem}[Existence and Uniqueness of stationary MFG solution] \label{thm1_stat} Given Assumptions \ref{policy_assumption_stat} and \ref{population_assumption_stat}, and assume $d_1d_2+d_3< 1$. Then there exists a unique stationary NE  to \eqref{mfg}.
\end{theorem}

\section{Additional comments on assumptions} 
As mentioned in the main text, the single player side Assumption \ref{policy_assumption} and its counterpart  Assumption \ref{policy_assumption_stat} for the stationary version correspond to the feedback regularity condition in the classical MFG literature. Here we add some comments on the population side Assumption \ref{population_assumption} and its stationary version Assumption \ref{population_assumption_stat}. For simplicity and clarity, let us consider the stationary case with finite state and action spaces. Then we have the following result. 
\begin{lemma}\label{assumption2_exp}
Suppose that $\max_{s,a,\mathcal{L},s'}P(s'|s,a,\mathcal{L})\leq c_1$, and that $P(s'|s,a,\cdot)$ is $c_2$-Lipschitz in $W_1$, \textit{i.e.},
\begin{equation}
|P(s'|s,a,\mathcal{L}_1)-P(s'|s,a,\mathcal{L}_2)|\leq c_2W_1(\mathcal{L}_1,\mathcal{L}_2).
\end{equation}
Then in Assumption \ref{population_assumption_stat}, $d_2$ and $d_3$ can be chosen as 
\begin{equation}
d_2=\frac{2\text{diam}(\mathcal{S})\text{diam}(\mathcal{A})|\mathcal{S}|c_1}{d_{\min}(\mathcal{A})}
\end{equation}
 and $d_3=\frac{\text{diam}(\mathcal{S})\text{diam}(\mathcal{A})c_2}{2}$, respectively.
\end{lemma}

Lemma \ref{assumption2_exp} provides an explicit characterization of the population side assumptions based only on the boundedness and Lipschitz properties of the transition dynamics $P$. In particular, $c_1$ becomes smaller when the transition dynamics becomes more diverse and the state space becomes larger.

\begin{proof} (Lemma \ref{assumption2_exp})
We begin by noticing that $\mathcal{L}'=\Gamma_2(\pi,\mathcal{L})$ can be expanded and computed as follows:
\begin{equation}
\mu'(s')=\sum\nolimits_{s\in\mathcal{S},a\in\mathcal{A}}\mu(s)P(s'|s,a,\mathcal{L})\pi(s,a),\quad \mathcal{L}'(s',a')=\mu'(s')\pi(s',a'),
\end{equation}
where $\mu$ is the state marginal distribution of $\mathcal{L}$.

Now by the inequalities (\ref{W1_tv}), we have
\begin{equation}
\begin{split}
W_1&(\Gamma_2(\pi_1,\mathcal{L}),\Gamma_2(\pi_2,\mathcal{L}))\leq \text{diam}(\mathcal{S}\times\mathcal{A})d_{TV}(\Gamma_2(\pi_1,\mathcal{L}),\Gamma_2(\pi_2,\mathcal{L}))\\
=&\dfrac{\text{diam}(\mathcal{S}\times\mathcal{A})}{2}\sum_{s'\in\mathcal{S},a'\in\mathcal{A}}\left|\sum_{s\in\mathcal{S},a\in\mathcal{A}}\mu(s)P(s'|s,a,\mathcal{L})\left(\pi_1(s,a)\pi_1(s',a')-\pi_2(s,a)\pi_2(s',a')\right)\right|\\
\leq& \dfrac{\text{diam}(\mathcal{S}\times\mathcal{A})}{2}\max_{s,a,\mathcal{L},s'}P(s'|s,a,\mathcal{L})\sum_{s,a,s',a'}\mu(s)(\pi_1(s,a)+\pi_2(s,a))|\pi_1(s',a')-\pi_2(s',a')|\\
\leq& \dfrac{\text{diam}(\mathcal{S}\times\mathcal{A})}{2}\max_{s,a,\mathcal{L},s'}P(s'|s,a,\mathcal{L})\sum_{s',a'}|\pi_1(s',a')-\pi_2(s',a')|\cdot (1+1)\\
=&2\text{diam}(\mathcal{S}\times\mathcal{A})\max_{s,a,\mathcal{L},s'}P(s'|s,a,\mathcal{L})\sum_{s'}d_{TV}(\pi_1(s'),\pi_2(s'))\\
\leq & \frac{2\text{diam}(\mathcal{S}\times\mathcal{A})\max_{s,a,\mathcal{L},s'}P(s'|s,a,\mathcal{L})|\mathcal{S}|}{d_{\min}(\mathcal{A})}D(\pi_1,\pi_2)=  \frac{2\text{diam}(\mathcal{S})\text{diam}(\mathcal{A})|\mathcal{S}|c_1}{d_{\min}(\mathcal{A})}D(\pi_1,\pi_2).
\end{split}
\end{equation}

Similarly, we have
\begin{equation}
\begin{split}
W_1&(\Gamma_2(\pi,\mathcal{L}_1),\Gamma_2(\pi,\mathcal{L}_2))\leq \text{diam}(\mathcal{S}\times\mathcal{A})d_{TV}(\Gamma_2(\pi,\mathcal{L}_1),\Gamma_2(\pi,\mathcal{L}_2))\\
=&\dfrac{\text{diam}(\mathcal{S}\times\mathcal{A})}{2}\sum_{s'\in\mathcal{S},a'\in\mathcal{A}}\left|\sum_{s\in\mathcal{S},a\in\mathcal{A}}\mu(s)\pi(s,a)\pi(s',a')\left(P(s'|s,a,\mathcal{L}_1)-P(s'|s,a,\mathcal{L}_2)\right)\right|\\
\leq & \dfrac{\text{diam}(\mathcal{S}\times\mathcal{A})}{2}\sum_{s,a,s',a'}\mu(s)\pi(s,a)\pi(s',a')\left|P(s'|s,a,\mathcal{L}_1)-P(s'|s,a,\mathcal{L}_2)\right|\\
\leq & \frac{\text{diam}(\mathcal{S})\text{diam}(\mathcal{A})c_2}{2}.
\end{split}
\end{equation}
This completes {\color{black}the} proof.
\end{proof}

\section{Proof of  Theorems \ref{thm1} and \ref{thm1_stat}}

For notational simplicity, we only present the proof for the stationary case (Theorem \ref{thm1_stat}). The proof of Theorems \ref{thm1} is the same with appropriate  notational changes. 

First by Definition \ref{nash2_stat} and the definitions of $\Gamma_i$ $(i=1,2)$, $(\pi,\mathcal{L})$ is a stationary NE  iff $\mathcal{L}=\Gamma(\mathcal{L})=\Gamma_2(\Gamma_1(\mathcal{L}),\mathcal{L})$ and $\pi=\Gamma_1(\mathcal{L})$, where $\Gamma(\mathcal{L})=\Gamma_2(\Gamma_1(\mathcal{L}),\mathcal{L})$. This indicates that for any $\mathcal{L}_1,\mathcal{L}_2\in\mathcal{P}(\mathcal{S}\times\mathcal{A})$,
\begin{equation}
\begin{split}
    &W_1(\Gamma(\mathcal{L}_1),\Gamma(\mathcal{L}_2))=W_1(\Gamma_2(\Gamma_1(\mathcal{L}_1),\mathcal{L}_1),\Gamma_2(\Gamma_1(\mathcal{L}_2),\mathcal{L}_2))\\
   &\leq W_1(\Gamma_2(\Gamma_1(\mathcal{L}_1),\mathcal{L}_1),\Gamma_2(\Gamma_1(\mathcal{L}_2),\mathcal{L}_1))+W_1(\Gamma_2(\Gamma_1(\mathcal{L}_2),\mathcal{L}_1),\Gamma_2(\Gamma_1(\mathcal{L}_2),\mathcal{L}_2))\\
   &\leq (d_1d_2+d_3)W_1(\mathcal{L}_1,\mathcal{L}_2).
    \end{split}
\end{equation}
And since $d_1d_2+d_3\in[0,1)$, by the Banach fixed-point theorem, we conclude that there exists a unique fixed-point of $\Gamma$, or equivalently, a unique stationary MFG solution to \eqref{mfg}.

\section{Proof of Theorem \ref{conv_AIQL}}

The proof  of Theorem \ref{conv_AIQL} relies on the following lemmas. 

\begin{lemma}[\cite{softmax}]\label{softmax-lip}
The softmax function is $c$-Lipschitz, \textit{i.e.}, $\|\textbf{softmax}_c(x)-\textbf{softmax}_c(y)\|_2\leq c\|x-y\|_2$ for any $x,~y\in\mathbb{R}^n$.
\end{lemma}

{\color{black} Notice that for a finite set $\mathcal{X}\subseteq\mathbb{R}^k$ and any two (discrete) distributions $\nu,~\nu'$ over $\mathcal{X}$, we have
\begin{equation}
\begin{split}
W_1(\nu,\nu')&\leq \text{diam}(\mathcal{X})d_{TV}(\nu,\nu')=\frac{\text{diam}(\mathcal{X})}{2}\|\nu-\nu'\|_1\leq \frac{\text{diam}(\mathcal{X})}{2}\|\nu-\nu'\|_2,
\end{split}
\end{equation}
where in computing the $\ell_1$-norm, $\nu,~\nu'$ are viewed as vectors of length $|\mathcal{X}|$. 

Hence Lemma \ref{softmax-lip} implies that for any $x,~y\in\mathbb{R}^{|\mathcal{X}|}$, when $\textbf{softmax}_c(x)$ and $\textbf{softmax}_c(y)$ are viewed as probability distributions over $\mathcal{X}$, we have 
\[
W_1(\textbf{softmax}_c(x),\textbf{softmax}_c(y))\leq \frac{\text{diam}(\mathcal{X})c}{2}\|x-y\|_2\leq \frac{\text{diam}(\mathcal{X})\sqrt{|\mathcal{X}|}c}{2}\|x-y\|_{\infty}.
\]}

\begin{lemma}\label{soft-arg-diff}
The distance between the softmax and the argmax mapping is bounded by
\[
\|\textbf{softmax}_c(x)-\textbf{argmax-e}(x)\|_2\leq 2n\exp(-c\delta),
\]
where {\color{black}$\delta=x_{\max}-\max_{x_j<x_{\max}}x_j$, $x_{\max}=\max_{i=1,\dots,n}x_i$, and}  $\delta:=\infty$ when all $x_j$ are equal.
\end{lemma}

{\color{black}
Similar to Lemma \ref{softmax-lip}, Lemma \ref{soft-arg-diff} implies that 
for any $x\in\mathbb{R}^{|\mathcal{X}|}$, viewing $\textbf{softmax}_c(x)$  as probability distributions over $\mathcal{X}$ leads to  
\[
W_1(\textbf{softmax}_c(x),\textbf{argmax-e}(x))\leq \text{diam}(\mathcal{X})|\mathcal{X}|\exp(-c\delta).
\]

\begin{proof}[Proof of Lemma~\ref{soft-arg-diff}]
Without loss of generality, assume that $x_1=x_2=\dots= x_m=\max_{i=1,\dots,n}x_i=x^\star>x_j$ for all $m<j\leq n$. Then \[
\textbf{argmax-e}(x)_i=
\begin{cases}
\frac{1}{m}, & i \leq m,\\
0, & otherwise.
\end{cases}
\]
\[
\textbf{softmax}_c(x)_i=
\begin{cases}
\frac{e^{cx^\star}}{me^{cx^\star}+\sum_{j=m+1}^ne^{cx_j}}, & i\leq m, \\
\frac{e^{cx_i}}{me^{cx^\star}+\sum_{j=m+1}^ne^{cx_j}}, & otherwise.
\end{cases}
\]
Therefore 
\begin{equation*}
\begin{split}
\|\textbf{soft}&\textbf{max}_c(x)-\textbf{argmax-e}(x)\|_2\leq \|\textbf{softmax}_c(x)-\textbf{argmax-e}(x)\|_1\\
=&m\left(\frac{1}{m}-\frac{e^{cx^\star}}{me^{cx^\star}+\sum_{j=m+1}^ne^{cx_j}}\right)+\frac{\sum_{i=m+1}^ne^{cx_i}}{me^{cx^\star}+\sum_{j=m+1}^ne^{cx_j}}\\
=& \frac{2\sum_{i=m+1}^ne^{cx_i}}{me^{cx^\star}+\sum_{i=m+1}^ne^{cx_i}}= \frac{2\sum_{i=m+1}^ne^{-c\delta_i}}{m+\sum_{i=m+1}^ne^{-c\delta_i}}\\
\leq &\frac{2}{m}\sum_{i=m+1}^ne^{-c\delta_i}\leq \frac{2(n-m)}{m}e^{-c\delta}\leq  2ne^{-c\delta},
\end{split}
\end{equation*}
with $\delta_i=x_i-x^\star$.
\end{proof}


\begin{lemma}[\cite{Q-rate}]\label{Q-finite-bd} For an MDP, say $\mathcal{M}$, suppose that the Q-learning algorithm takes step-sizes 
\[
\beta_t(s,a)=
\begin{cases}
|\#(s,a,t)+1|^{-h}, & (s,a)=(s_t,a_t),\\
0, & \text{otherwise}.
\end{cases}
\]
with $h\in(1/2,1)$. Here $\#(s,a,t)$ is the number of times up to time $t$ that one visits the state-action pair $(s,a)$. Also suppose that the covering time of the state-action pairs is bounded by $L$ with probability at least {\color{black}$1-p$ for some $p\in(0,1)$}. Then $\|Q_{T^{\mathcal{M}}(\delta,\epsilon)}-Q^\star\|_{\infty}\leq \epsilon$  with probability at least $1-2\delta$. Here $Q_T$ is the $T$-th update in Q-learning, and $Q^\star$ is the (optimal) Q-function, given that
\[
\begin{split}
T^{\mathcal{M}}(\delta,&\epsilon)=\Omega\left(\left(\dfrac{L\log_p(\delta)}{\beta}\log\dfrac{V_{\max}}{\epsilon}\right)^{\frac{1}{1-h}}+\left(\dfrac{\left(L\log_p(\delta)\right)^{1+3h}V_{\max}^2\log\left(\frac{|\mathcal{S}||\mathcal{A}|V_{\max}}{\delta\beta\epsilon}\right)}{\beta^2\epsilon^2}\right)^{\frac{1}{h}}\right),
\end{split}
\]
where $\beta=(1-\gamma)/2$, $V_{\max}=R_{\max}/(1-\gamma)$, and $R_{\max}$ is an upper bound on the extreme difference between the expected rewards, \textit{i.e.}, $\max_{s,a,\mu}{r}(s,a,\mu)-\min_{s,a,\mu}{r}(s,a,\mu)\leq R_{\max}$.
\end{lemma} 

Here the covering time {\color{black}$L$} of a state-action pair sequence is defined to be the number of steps needed to visit all state-action pairs starting from any arbitrary state-action pair, and  {\color{black} $T^{\mathcal{M}}(\delta,\epsilon)$ is the number of inner iterations $T_k$ set in Algorithm~\ref{AIQL_MFG}. This will guarantee the convergence in Theorem  \ref{conv_AIQL}.}
{\color{black}Also notice that the $l_{\infty}$ norm above is defined in an element-wise sense, \textit{i.e.}, for $M\in\mathbb{R}^{|\mathcal{S}|\times|\mathcal{A}|}$, we have $\|M\|_\infty=\max_{s\in\mathcal{S},a\in\mathcal{A}}|M(s,a)|$.}

\begin{proof}[Proof of Theorem \ref{conv_AIQL}]
Define $\hat{\Gamma}_1^k(\mathcal{L}_k):=\textbf{softmax}_c\left(\hat{Q}^\star_{\mathcal{L}_k}\right)$. In the following, $\pi=\textbf{softmax}_c(Q_{\mathcal{L}})$ is understood as the policy $\pi$ with $\pi(s)=\textbf{softmax}_c(Q_{\mathcal{L}}(s,\cdot))$. Let $\mathcal{L}^\star$ be the population state-action pair in a stationary NE of \eqref{mfg}. Then $\pi_k=\hat{\Gamma}_1^k(\mathcal{L}_k)$.
Denoting $d:=d_1d_2+d_3$,  we see 
\begin{equation*}
    \begin{split}
        W_1(\tilde{\mathcal{L}}_{k+1}&,\mathcal{L}^\star)=W_1(\Gamma_2(\pi_k,\mathcal{L}_k),\Gamma_2(\Gamma_1(\mathcal{L}^\star),\mathcal{L}^\star))\\
        \leq& W_1(\Gamma_2(\Gamma_1(\mathcal{L}_k),\mathcal{L}_k),\Gamma_2(\Gamma_1(\mathcal{L}^\star),\mathcal{L}^\star))+W_1(\Gamma_2(\Gamma_1(\mathcal{L}_k),\mathcal{L}_k),\Gamma_2(\hat{\Gamma}_1^k(\mathcal{L}_k),\mathcal{L}_k))\\
        \leq& W_1(\Gamma(\mathcal{L}_k),\Gamma(\mathcal{L}^\star))+d_2D(\Gamma_1(\mathcal{L}_k),\hat{\Gamma}_1^k(\mathcal{L}_k))\\
        \leq & (d_1d_2+d_3)W_1(\mathcal{L}_k,\mathcal{L}^\star)+d_2D(\textbf{argmax-e}(Q_{\mathcal{L}_k}^\star),\textbf{softmax}_c(\hat{Q}_{\mathcal{L}_k}^\star))\\
        \leq& dW_1(\mathcal{L}_k,\mathcal{L}^\star)+d_2D(\textbf{softmax}_c(\hat{Q}_{\mathcal{L}_k}^\star),\textbf{softmax}_c(Q_{\mathcal{L}_k}^\star))\\
        &+d_2D(\textbf{argmax-e}(Q_{\mathcal{L}_k}^\star),\textbf{softmax}_c(Q_{\mathcal{L}_k}^\star))\\
        \leq & dW_1(\mathcal{L}_k,\mathcal{L}^\star)+\frac{cd_2\text{diam}(\mathcal{A})\sqrt{|\mathcal{A}|}}{2}\|\hat{Q}_{\mu_k}^\star-Q_{\mu_k}^\star\|_{\infty}\\
        &+d_2D(\textbf{argmax-e}(Q_{\mathcal{L}_k}^\star),\textbf{softmax}_c(Q_{\mathcal{L}_k}^\star)).
    \end{split}
\end{equation*}
Then since $\mathcal{L}_k\in S_{\epsilon}$ by the projection step, Lemma \ref{soft-arg-diff}, and Lemma \ref{Q-finite-bd} with the choice of $T_k=T^{\mathcal{M}_\mu}(\delta_k,\epsilon_k)$),  we have,  with probability at least $1-2\delta_k$, 
\begin{equation}
W_1(\tilde{\mathcal{L}}_{k+1},\mathcal{L}^\star)\leq  dW_1(\mathcal{L}_k,\mathcal{L}^\star)+\frac{cd_2\text{diam}(\mathcal{A})\sqrt{|\mathcal{A}|}}{2}\epsilon_k+d_2\text{diam}(\mathcal{A})|\mathcal{A}|e^{-c\phi(\epsilon)}.
\end{equation}

Finally, it is clear that with probability at least $1-2\delta_k$,
\begin{equation*}
    \begin{split}
   W_1(\mathcal{L}_{k+1},\mathcal{L}^\star)&\leq W_1(\tilde{\mathcal{L}}_{k+1},\mathcal{L}^\star)+ W_1(\tilde{\mathcal{L}}_{k+1},\textbf{Proj}_{S_{\epsilon}}(\tilde{\mathcal{L}}_{k+1})) \\
    &\leq dW_1(\mathcal{L}_k,\mathcal{L}^\star)+\frac{cd_2\text{diam}(\mathcal{A})\sqrt{|\mathcal{A}|}}{2}\epsilon_k+d_2\text{diam}(\mathcal{A})|\mathcal{A}|e^{-c\phi(\epsilon)}+\epsilon.
    \end{split}
\end{equation*}
By telescoping, this implies that with probability at least $1-2\sum_{k=0}^{{\color{black}K-1}}\delta_k$, 
\begin{equation}
\begin{split}
W_1(\mathcal{L}_K, \mathcal{L}^\star)\leq &d^KW_1(\mathcal{L}_0,\mathcal{L}^\star)+\frac{cd_2\text{diam}(\mathcal{A})\sqrt{|\mathcal{A}|}}{2}\sum_{k=0}^{K-1}d^{K-k}\epsilon_k \\
&+\dfrac{(d_2\text{diam}(\mathcal{A})|\mathcal{A}|e^{-c\phi(\epsilon)}+\epsilon)(1-d^{\color{black}K})}{1-d}.
\end{split}
\end{equation}

Since $\epsilon_k$ is summable,  hence $\sup_{k\geq 0}\epsilon_k<\infty$, 
$\sum_{k=0}^{K-1}d^{K-k}\epsilon_k\leq \dfrac{\sup_{k\geq 0}\epsilon_k}{1-d}d^{\lfloor(K-1)/2\rfloor}+\sum_{k=\lceil(K-1)/2\rceil}^{\infty}\epsilon_k$.

Now plugging in $K=K_{\epsilon,\eta}$, with the choice of $\delta_k$ and $c=\frac{\log(1/\epsilon)}{\phi(\epsilon)}$, and noticing that $d\in[0,1)$, it is clear that  with probability at least $1-2\delta$, 
\begin{equation}
\begin{split}
W_1(\mathcal{L}_{K_{\epsilon,\eta}},\mathcal{L}^\star)\leq &d^{K_{\epsilon,\eta}}W_1(\mathcal{L}_0,\mathcal{L}^\star)\\
&+\frac{cd_2\text{diam}(\mathcal{A})\sqrt{|\mathcal{A}|}}{2}\left( \dfrac{\sup_{k\geq 0}\epsilon_k}{1-d}d^{\lfloor(K_{\epsilon,\eta}-1)/2\rfloor}+\sum_{k=\lceil(K_{\epsilon,\eta}-1)/2\rceil}^{\infty}\epsilon_k\right)\\
&+\dfrac{(d_2\text{diam}(\mathcal{A})|\mathcal{A}|+1)\epsilon}{1-d}.
\end{split}
\end{equation}

Setting $\epsilon_k=(k+1)^{-(1+\eta)}$, then when $K_{\epsilon,\eta}\geq 2(\log_d(\epsilon{\color{black}/c})+1)$, 
\[
 \dfrac{\sup_{k\geq 0}\epsilon_k}{1-d}d^{\lfloor(K_{\epsilon,\eta}-1)/2\rfloor}\leq \frac{\epsilon{\color{black}/c}}{1-d}.
\] 
Similarly, when $K_{\epsilon,\eta}\geq 2(\eta\epsilon{\color{black}/c})^{-1/\eta}$,  $\sum_{k=\left\lceil\frac{K_{\epsilon,\eta}-1}{2}\right\rceil}^{\infty}\epsilon_k\leq \epsilon{\color{black}/c}$.

Finally, when $K_{\epsilon,\eta}\geq \log_d(\epsilon/(\text{diam}(\mathcal{S})\text{diam}(\mathcal{A})))$, 
$d^{K_{\epsilon,\eta}}W_1(\mathcal{L}_0,\mathcal{L}^\star)\leq \epsilon$,
since $W_1(\mathcal{L}_0,\mathcal{L}^\star)\leq \text{diam}(\mathcal{S}\times\mathcal{A}){\color{black}=\text{diam}(\mathcal{S})\text{diam}(\mathcal{A})}$. 

In summary, if $K_{\epsilon,\eta}=\lceil 2\max\{(\eta\epsilon{\color{black}/c})^{-1/\eta}$, $\log_d(\epsilon/\max\{\text{diam}(\mathcal{S})\text{diam}(\mathcal{A}),{\color{black}c}\})+1\}\rceil$, then with probability at least $1-2\delta$,  
\[
\begin{split}
&W_1(\mathcal{L}_{K_{\epsilon,\eta}},\mathcal{L}^\star)\leq \left(1+\frac{d_2\text{diam}(\mathcal{A})\sqrt{|\mathcal{A}|}(2-d)}{2(1-d)}+\dfrac{(d_2\text{diam}(\mathcal{A})|\mathcal{A}|+1)}{1-d}\right)\epsilon=O(\epsilon).
\end{split}
\]

Finally, plugging in $\epsilon_k$ and $\delta_k$ into $T^{\mathcal{M}_L}(\delta_k,\epsilon_k)$, and noticing that $k {\color{black} \leq }  K_{\epsilon,\eta}$ and $\sum_{k=0}^{K_{\epsilon,\eta}-1}(k+1)^{\alpha}\leq \frac{K_{\epsilon,\eta}^{\alpha+1}}{\alpha+1}$, we immediately arrive at 
\begin{equation*}\label{Tbound_complex}
\begin{split}
T= O&\left(\left(\log(K_{\epsilon,\eta}/\delta)\right)^{\frac{1}{1-h}}K_{\epsilon,\eta}\left(\log K_{\epsilon,\eta}\right)^{\frac{1}{1-h}}
+\left(\log(K_{\epsilon,\eta}/\delta)\right)^{\frac{1}{h}+3}\frac{K_{\epsilon,\eta}^{1+\frac{2(1+\eta)}{h}}}{1+\frac{2(1+\eta)}{h}}\left(\log (K_{\epsilon,\eta}/\delta)\right)^{\frac{1}{h}}\right).
\end{split}
\end{equation*}
By further relaxing $\eta$ to $1$ and merging the terms, (\ref{Tbound}) follows.
\end{proof}

\section{Naive algorithm}

 The Naive iterative algorithm (Algorithm \ref{AQL_MFG}) is to  replace Step A in the three-step fixed-point approach of GMFGs with Q-learning iterations. The limitation of this Naive algorithm has been discussed in the main text {\color{black}(Step 1, Section \ref{AIQL})} and {\color{black}empirically verified in Section \ref{experiments}} (Figure \ref{fig:naive}). 
\begin{algorithm}
  \caption{\textbf{Alternating Q-learning for GMFGs (Naive)}}
  \label{AQL_MFG}
\begin{algorithmic}[1]
  \STATE \textbf{Input}: Initial population state-action pair $L_0$
 \FOR {$k=0, 1, \cdots$}
  \STATE Perform Q-learning to find the Q-function $Q_k^\star(s,a)=Q_{L_k}^\star(s,a)$ of an MDP with dynamics $P_{L_k}(s'|s,a)$ and rewards $r_{L_k}(s,a)$.
  \STATE Solve $\pi_k\in \Pi$ with $\pi_k(s)=\textbf{argmax-e}\left(Q_k^\star(s,\cdot)\right)$.
  \STATE Sample $s\sim \mu_k$, where $\mu_k$ is the population state marginal of $L_k$, and obtain $L_{k+1}$ from $\mathcal{G}(s,\pi_k,L_k)$.
\ENDFOR
\end{algorithmic}
\end{algorithm}

\section{GMF-V}\label{GMF-V}
GMF-V, briefly mentioned in Section \ref{AIQL}, is the value-iteration version of our  main algorithm GMF-Q. GMF-V applies to the GMFG setting with fully known transition dynamics $P$ and rewards $r$.

\begin{algorithm}[H]
  \caption{\textbf{Value Iteration for GMFGs (GMF-V)}}
  \label{VI}
\begin{algorithmic}[1]
  \STATE \textbf{Input}: Initial $L_0$, tolerance $\epsilon>0$.
 \FOR {$k=0, 1, \cdots$}
  \STATE Perform value iteration for $T_k$ iterations to find the approximate Q-function $Q_{L_k}$ and value function $V_{L_k}$: 
   \FOR {$t=1,2,\cdots,T_k$}
   \FOR {all $s \in \mathcal{S}$ and $s \in \mathcal{A}$} 
   \STATE $ Q_{L_k}(s,a) \leftarrow \mathbb{E}[r(s,a,L_k)]+\gamma \sum_{s^{\prime}}P(s^{\prime}|s,a,L_k)V_{L_k}(s^{\prime})$
   \STATE $V_{L_k}(s)\leftarrow\max_a Q_{L_k}(s,a)$
   
   \ENDFOR
   \ENDFOR
  
  \STATE Compute a policy $\pi_k\in\Pi$: \\ ${ \hspace{1cm}\pi_k(s)=\textbf{softmax}_c({Q}_{L_k}(s,\cdot))}$.
  \STATE Sample $s\sim \mu_k$, where $\mu_k$ is the population state marginal of $L_k$, and obtain $\tilde{L}_{k+1}$ from $\mathcal{G}(s,\pi_k,L_k)$.
  \STATE Find $L_{k+1}=\textbf{Proj}_{S_{\epsilon}}(\tilde{L}_{k+1})$
\ENDFOR
\end{algorithmic}
\end{algorithm}

\section{More details for the experiments}
\subsection{Competition intensity index $M$.}\label{intensity}
In the experiment, the competition index $M$ is interpreted and implemented as the number of selected players in each auction competition. That is, in each round, $M-1$ players will be randomly
 selected from the population to compete with the {\it representative} advertiser for the auction. Therefore, the population distribution $\mathcal{L}_t$, the winner indicator {\color{black}$w_t^M$}, and second-best price $a_t^M$ all depend on $M$. This parameter $M$ is also referred to as the {\it auction thickness} in the auction literature \cite{IJS2011}.
\subsection{Adjustment for Algorithm MF-Q.}
 For MF-Q, \cite{YLLZZW2018} assumes all $N$ players have a joint state $s$. In the auction experiment, we make the following adjustment for MF-Q for computational efficiency and model comparability: each player $i$ makes decision based on her own private state and table $Q^i$ is a functional of $s^i$, $a^i$ and $\frac{\sum_{j \neq i}a^j}{N-1}$.


}
\end{document}